\documentclass[final,onefignum,onetabnum]{siamonline220329}
\usepackage{braket,amsfonts}
\usepackage{xspace}
\usepackage{float}
\usepackage{etoolbox}
\usepackage{hyperref}
\usepackage{enumitem}%

\input{def}
\pgfplotsset{compat=1.16}

\usepackage{lipsum}
\usepackage{graphicx}
\usepackage{epstopdf}

\ifpdf
  \DeclareGraphicsExtensions{.eps,.pdf,.png,.jpg}
\else
  \DeclareGraphicsExtensions{.eps}
\fi

\usepackage{enumitem}
\setlist[enumerate]{leftmargin=.5in}
\setlist[itemize]{leftmargin=.5in}

\def\addressncsu{Department of Mathematics, North Carolina State University,
  Raleigh, NC, USA}

\def\sandia{Sandia National Laboratories}

\author{
Madhusudan Madhavan\thanks{\addressncsu}, \and 
Alen Alexanderian\thanks{\addressncsu}, \and 
Arvind K. Saibaba\thanks{\addressncsu}, \and 
Bart van Bloemen Waanders\thanks{\sandia}, \and
Rebekah D. White\thanks{\sandia}
}

\graphicspath{{./figs/}} %

\headers{Control-oriented optimal sensor placement}{M.~Madhavan et al.}
\title{A control-oriented approach to optimal sensor placement\thanks{\vspace{-\baselineskip}\funding{Madhavan was supported by the National Science
Foundation through the award DMS-1745654. Saibaba was supported, in part, by the
Department of Energy through the award DE-SC0023188.
The work of Alexanderian and van Bloemen Waanders was supported in part by 
the U.S. Department of
Energy, Office of Science, Office of Advanced Scientific Computing Research
Field Work Proposal Number 23-02526.
Madhavan and Saibaba were also supported, in part, by the National Science Foundation through the award DMS-2411198.
}}}

\begin{document}
\maketitle

\begin{abstract}
We propose a control-oriented optimal experimental design (cOED) approach for
linear PDE-constrained Bayesian inverse problems. 
In particular, we consider optimal control 
problems with uncertain parameters that need to be 
estimated by solving an inverse problem, which in turn requires measurement data.
We consider the case where 
data is collected at a set of sensors.  
While classical Bayesian
OED techniques provide experimental designs (sensor placements) that minimize the 
posterior uncertainty in the inversion parameter, these designs are not 
tailored to the demands of the optimal control problem. 
In the present control-oriented setting, we prioritize the designs that minimize the
uncertainty in the state variable being controlled or the control objective. 
We propose a mathematical  
framework for uncertainty quantification and cOED for
parameterized PDE-constrained optimal control problems with linear dependence to
the control variable and the inversion parameter. We also 
present scalable computational methods for computing control-oriented sensor 
placements and for quantifying the uncertainty in the control objective.
Additionally, we present illustrative numerical results in the context of a 
model problem motivated by heat transfer applications.

\end{abstract}



\section{Introduction}
\label{sec:intro}

Optimal control of systems governed by partial differential equations (PDEs) 
is a commonly occurring problem in the sciences and engineering. Examples of 
control problems can be found in wildfires, sub-surface petroleum reservoirs, 
HVAC systems, experimental fusion energy, and additive manufacturing 
\cite{GunzburgerFlow2002,Troltzsch2010OCPDE, Kessel1990LinearOC, HibaAdditiveOCP2024}. A key challenge in such optimal control 
problems is the presence of uncertain parameters in the governing PDEs. 
Examples of such parameters include coefficient functions, 
volume of boundary source terms, or initial conditions.

Consider, for example, a control problem of the form
\begin{equation}
    \begin{aligned}
       &\min_{z} \ \ctrlObj(u(z), z; m) \\
       \text{where}
       \\
       &\quad \begin{cases}
          \displaystyle \frac{\partial u}{\partial t} = \mc{G}(u, z, m),\\ u(\cdot, 0) = u_0.
       \end{cases}
    \end{aligned}
    \label{eq:parameterized_governing_PDE}
 \end{equation}
Here, $u$ is a state variable, $z$ is a control, $m$ is an uncertain parameter,
and $\ctrlObj$ is a suitably defined reduced cost functional.  In the present
work, we consider the case where the goal of the control problem is to steer the
state variable toward a desired target terminal state.  Note that the resulting
optimal control problem is parameterized by the uncertain parameter $m$. 

A key challenge in such parameterized optimal control problems is that the computed 
optimal control might depend strongly on the value of the uncertain parameter
$m$. Therefore, having access to a reliable estimate of $m$ is important. 
Often, we can estimate $m$ by solving an inverse problem. 
In the present
work, we consider a Bayesian paradigm~\cite{Tarantola05,Stuart10}, 
in which we use a model and data to find
a posterior distribution law for $m$. Not only does this provide a useful point
estimate for $m$, given by the maximum a posteriori probability (MAP) point, it
also enables quantifying the uncertainty in the optimal control or the control
objectives. In a nutshell, to enable solving the optimal control problem 
with quantified uncertainties, we need to perform the following steps:
\begin{enumerate}[label=(\roman*)]
\item acquire measurement data;
\item solve an inverse problem to estimate $m$; and
\item solve the optimal control problem.
\end{enumerate}
Note, however, that step (i) often involves a costly process.
Namely, acquiring experimental data can be
time-consuming and is subject to physical and/or budgetary constraints. Thus, it is
crucial to collect data optimally. This requires solving an optimal
experimental design (OED) problem~\cite{AtkinsonDonev92,Pukelsheim06,Ucinski05}.
In the present work, we consider a Bayesian approach to OED~\cite{Chaloner1995BayesianED}.

Traditional approaches to Bayesian OED provide techniques for minimizing
uncertainty in the inversion parameter $m$.  However, these approaches are not
well-suited for problems in which the estimation of the inversion parameter is
merely an intermediate step.  Specifically, finding a data acquisition strategy
that is tailored to parameterized optimal control problems requires a
\emph{control-oriented OED (cOED)} framework. This is the primary focus of this
article.  
In particular, cOED involves coupling three different problems: 
(i) an optimal control problem; (ii) an inverse problem; and (iii) an optimal 
sensor placement problem. 
A schematic indicating the different components of such a framework 
is depicted in~\Cref{fig:diag_simple}. 

\begin{figure}
   \centering
   \includegraphics[width=1\textwidth]{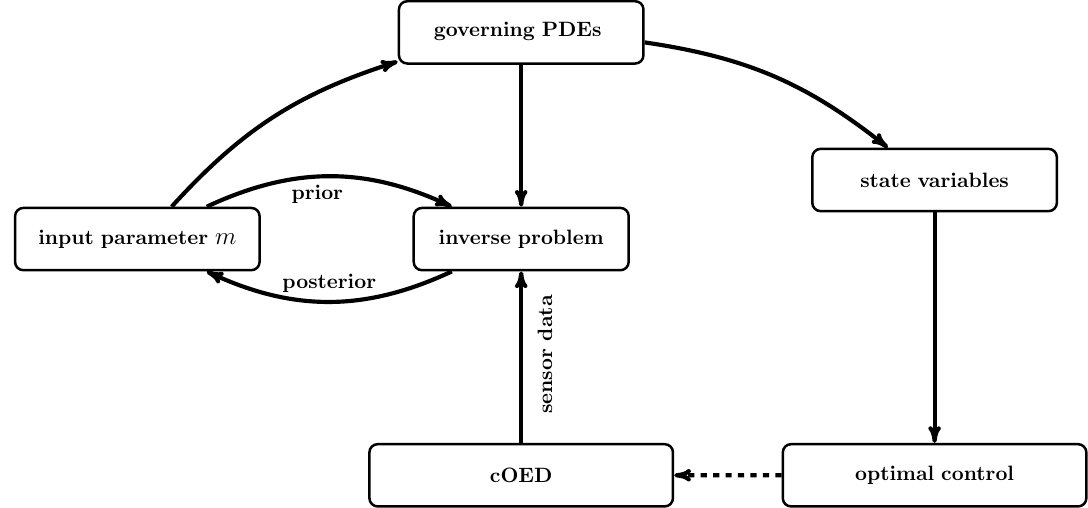}
   \caption{The various aspects of optimal control, inversion, and design of experiments.
   The optimal control problem informs the definition of a cOED criterion. Subsequently, 
   we solve 
   the cOED problem (to obtain sensor data), the inverse problem
   (to estimate $m$), and the optimal control problem. In the diagram, the governing PDEs 
   include the PDEs governing the inverse problem and those governing the 
   control problem.}
   \label{fig:diag_simple}
\end{figure}

In the present work, we take a foundational step in developing a mathematical
and computational framework for inversion, design of experiments, and control
for systems governed by PDEs with infinite-dimensional uncertain parameters.  We
consider the case where $\mc{G}$ in~\cref{eq:parameterized_governing_PDE} is
linear in $u$, $z$, and $m$.  This setup covers problems involving the
control of systems governed by linear PDEs where (i) we seek controls in
boundary or volume source terms; and (ii) have uncertainties in the remaining
source terms in the governing equations.  We focus on problems in which the
inversion parameter $m$ is estimated using measurement data collected at a
set of sensors. 
Moreover, 
we assume a linear (or linearized) model governing the inverse problem.
To elucidate the components of the coupled
optimal control, inversion, and cOED formulations, we present a motivating model
problem formulation, involving a heat transfer application, in~\Cref{sec:motivation}.
The requisite background regarding inversion, design of experiments, and 
optimal control in the class of problems under study is outlined in 
see~\Cref{sec:background_material}.

\boldheading{Related work}
Design of experiments for large-scale Bayesian inverse problems has been subject
to intense research activity in recent years; see the review
articles~\cite{Alexanderian21,Huan2024OptimalED} for a survey of the literature.
Optimal control of systems governed by PDEs under uncertainty 
has also been investigated in various studies; see, e.g., 
\cite{BorziSchulzSchillingsEtAl10,GunzburgerMing11,TieslerKirbyXiuEtAl12,KouriHeinkenschlossRidzalEtAl13,
AlexanderianPetraStadlerEtAl17,ChenGhattas21}. The present work takes a foundational step 
in bridging the gap between 
optimal design of experiments for 
inverse problems governed by PDEs and optimal control under uncertainty.
The proposed approach is related to recent efforts on
\emph{goal-oriented} 
OED~\cite{HerzogRiedelUcinski18,Li19,ButlerJakemanWildey20,ZhongShenCatanachEtAl24}.
Namely, cOED can be considered as a type of goal-oriented OED, with a specific
goal that is informed by an optimal control problem.  The developments in the
present work are related to the
works~\cite{AttiaAlexanderianSaibaba18,WuChenGhattas23}, which consider
goal-oriented optimal design of infinite-dimensional linear Bayesian inverse
problems.  While there has been some efforts on relating optimal control to the
design of experiments \cite{PronzatoOptimal2008, HjalmarssonExperiment2005,
GeversIdentification2005}, to our knowledge, there have been no efforts on a
control-oriented approach to the optimal design of experiments for Bayesian
inverse problems.  

\boldheading{Our approach and contributions}
We consider the average variance in the terminal state, $u(\cdot,
T_\text{final})$, as our choice of design criterion.  We call this the
\emph{cOED criterion}.  
This criterion takes the form of a weighted A-optimality criterion that combines
constructs in the inverse and optimal control problems.
Upon finding an optimal sensor placement, we can solve the
inverse problem to estimate $m$. Subsequently, the optimal control problem is
solved with $m$ taken as the MAP point.  However, the terminal state and 
the optimal control objective are still uncertain, due to the
uncertainty in $m$.  To address this, we also develop a mathematical framework for
quantifying the uncertainty in the optimal control objective.  This
includes expressions for mean and variance of the control objective as
well as a concentration bound. The latter
provides important qualitative insight regarding the importance of optimizing
our cOED criterion.  Namely, we show that the cOED criterion can
be used to bound the probability of large deviations in the control objective.
The proposed framework for quantifying the uncertainty in the terminal state and
the optimal control objective is detailed in~\Cref{sec:dataDrivenOC}.

The expressions for the cOED criterion and measures of uncertainty in the
optimal control objective involve traces of high-dimensional operators that are
defined only implicitly.  To enable implementations in large-scale applications,
we provide fast and accurate computational methods for the various steps of the
proposed control-oriented optimal experimental design and uncertainty
quantification framework. These include fast methods for computing the cOED
objective and the mean and variance of the control objective. The proposed
approach builds on randomized trace estimators and structure-exploiting
approaches for Bayesian inversion in a function space. The proposed
computational framework for cOED and uncertainty quantification for parameterized
optimal control problems is presented in~\Cref{sec:computational_methods}. 

The key contributions of this article are
\begin{itemize}
\item An end-to-end mathematical framework for uncertainty quantification and cOED  for
parameterized PDE-constrained optimal control problems with linear dependence to
the control variable and the inversion parameter;

\item computational methods for fast evaluation of the cOED objective
as well as measures of uncertainty in optimal state trajectory and 
optimal control objective; and 

\item concentration bounds for the control objective, relating the 
proposed cOED objective to the deviations in the control objective.

\end{itemize}
Additionally, we elaborate the proposed framework for cOED
in~\Cref{sec:num_res}, in the context of the heat transfer problem discussed 
in~\Cref{sec:motivation}. Our
extensive numerical results demonstrate the effectiveness of the proposed
strategy.  In~\Cref{sec:conclusion}, we present concluding remarks and discuss
the limitations of the present study and potential directions for 
future work.  

\section{Motivating application and model problem}
\label{sec:motivation}

In this section, we consider a 
model problem to 
illustrate the types of interconnections between optimal control, parameter inversion,
and optimal experimental design (OED) that might arise in applications. The following example builds on well-known example found in inverse problem literature \cite{akcelik2005dynamic,petra2011model,flath2011fast,kim2023hippylib}, as well as OED literature \cite{alexanderian2014a,alexanderian2018efficient,WuChenGhattas23}.
Consider heat transfer in a bounded two-dimensional domain $\Omega$ 
depicted in \Cref{fig:dom}. 
Suppose we have control over a source term $z$ that acts within a particular
subdomain $\Omega_c \subset \Omega$, which we refer to as the control region. Our goal is 
to steer the temperature distribution within $\Omega$
by controlling the temperature within $\Omega_c$. 
\begin{figure}[ht]
   \centering
   \includegraphics[width=.3\textwidth]{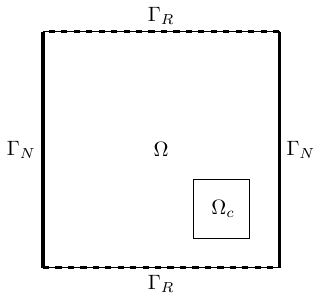}
   \setlength{\belowcaptionskip}{-8pt}  
   \caption{The domain of the heat transfer problem \cref{eq:noControlHeat}.}
   \label{fig:dom}
\end{figure}
To make matters concrete, we assume the system is governed by the following
advection diffusion equation.
\begin{equation}\label{eq:controlProblem}
\begin{alignedat}{2}
    \state_t -\diffParam \Delta \state + \bs{v} \cdot \nabla u &= \invParamCtn + z \, \mathds{1}_{\ctrlDom} 
    &&\quad\text{in } \invDom \times (0, T),\\
    \state(0, \bs{x}) &= \invState(\bs{x}; m)
    &&\quad\text{in } \invDom, \\
    \nabla \state \cdot \vec{n} &= 0 
    &&\quad\text{on } \GN \times (0, T), \\
    \nabla \state \cdot \vec{n} &= \gamma_h(\state - \gamma_a)
    &&\quad\text{on } \GR \times (0, T).
  \end{alignedat}
\end{equation}
Here, $\kappa$ is a diffusion constant, $\bs{v}$ is a (steady) velocity field,
and $\gamma_h, \gamma_a$ are boundary coefficients.  The system is fully
insulated at $\BdrySides$, and $\BdryTops$ experiences heat exchange from
the ambient temperature $\gamma_a$ with heat transfer rate $\gamma_h$. 

Let us consider the volume source terms in~\cref{eq:controlProblem}.
The first term, $m$, is an uncertain background source term, and the second one 
incorporates a control function $z:[0, T] \to \R$, where $T$ is a final time. 
Here, $\mathds{1}_{\ctrlDom}$ is the indicator function of the control region
$\Omega_c$.  We assume that we can control the temperature
uniformly within $\Omega_c$.  Furthermore, 
$m$ is assumed to be a function $m:\Omega\to\R$ that does not depend on time. Therefore, the combined source term 
in~\cref{eq:controlProblem} is the function $(\vec{x}, t)\mapsto m(\vec x) + \mathds{1}_{\ctrlDom}(\vec{x}) z(t)$.
Note that the initial state $\invState$ is also
parameterized by $m$. This is made precise shortly. 

The optimal control problem seeks to 
steer the state $u$ towards a target state
$\bar{u}$ by time $T$.  This is done by finding $u$ 
that solves
\begin{equation}\label{eq:OC_ctn}
  \min_{z \in L^2[0, T]} \Phi(z; m) \eqdef \frac12 \int_\Omega (u(\cdot, T) - \bar{u})^2 d\bs{x} + \frac{\beta}{2}\int_0^T z(t)^2 dt \\,
\end{equation}
where $\state$ is the solution of the problem~\cref{eq:controlProblem} 
and $\beta > 0$ is a regularization parameter.

We next discuss the uncertain source term $m$ in the above optimal control
problem.  
We assume that the temperature distribution within the room is in
steady state before the initial time, $t=0$. Specifically, we assume 
that the initial temperature distribution $\hat{u}$ satisfies   
\begin{equation}\label{eq:noControlHeat}
   \begin{alignedat}{2}
     -\diffParam \Delta \invState &= \invParamCtn 
      &&\quad\text{in } \invDom, \\
     \nabla \invState \cdot \vec{n} &= 0
       &&\quad\text{on } \GN, \\
       \nabla \invState \cdot \vec{n} &= \gamma_h(\invState - \gamma_a)
       &&\quad\text{on } \GR.
     \end{alignedat}
 \end{equation}
This equation describes the (steady) temperature distribution, before 
applying temperature control in $\Omega_c$. Notably, the governing physics in \cref{eq:noControlHeat} is different from that of \cref{eq:controlProblem}.

As noted previously, the source term $m$ is subject to uncertainty. We assume
that we can reduce this uncertainty by solving an inverse problem.  Suppose one
can collect sensor measurements of temperature at a few sensors.  We let
$\bs{y}$ denote the vector of measurements of $\hat{u}$.  In what follows, we
assume the following observation model,
\begin{equation}\label{eq:observation_model_V1}
   \obs = \mc{\hat{B}}\hat{u} + \vec{\eta}.
\end{equation}
Here, $\vec{\eta} \sim \GM{\vec{0}}{\noisecov}$ models measurement noise and
$\mc{\hat{B}}$ is an observation operator that maps $\hat{u}$ to the measurement space.
Note that the mapping $m\mapsto\hat{u}$ is an affine 
mapping of the form $\hat{u} = \mc{S}m + s$,
where $\mc{S}$ is a continuous linear transformation.
We thus restate~\cref{eq:observation_model_V1} as 
\begin{equation}
   \obs = \mc{F}\m + \vec{b} + \vec{\eta},
   \label{eq:observation_model}
\end{equation}
where $\mc{F} \eqdef \mc{\hat{B}} \circ \mc{S}$ and $\vec{b} \eqdef \mc{\hat{B}}s$. 
Combining the present observation model with prior knowledge regarding 
$m$, we can formulate a Bayesian inverse problem for estimating $m$; see 
\Cref{sec:ctn_inverse_problem}.

As noted above, we use sensor measurements of $\hat{u}$ to estimate $m$. 
Optimizing the placement of these sensors requires solving an OED problem; 
see~\Cref{sec:oed} for background concepts regarding OED. 
However, since the estimation of $m$ is only an intermediate step 
and the ultimate goal is solving the optimal control problem~\cref{eq:OC_ctn}, 
a cOED approach is needed. In~\Cref{sec:dataDrivenOC}, 
we detail our proposed cOED formulation, for the class of problems under study.

\section{Preliminaries}
\label{sec:background_material}
In this section, we discuss the necessary background for
infinite-dimensional Bayesian inverse problems, optimal control, and design 
of experiments.

\subsection{Notation}
\label{sec:notation}
For a Hilbert space $\fnsp{H}$, we denote the inner product by
$\ipg{\cdot}{\cdot}{\fnsp{H}}$ and the corresponding induced norm by
$\|\cdot\|_\fnsp{H}^2 = \ipg{\cdot}{\cdot}{\fnsp{H}}$. 
All Hilbert spaces considered in the present work are real separable 
Hilbert spaces.
For two Hilbert spaces $\fnsp{H}_1$ and
$\fnsp{H}_2$, we denote the space of bounded linear transformations from $\fnsp{H}_1$
to $\fnsp{H}_2$ by $\sL(\fnsp{H}_1, \fnsp{H}_2)$ and the space of bounded 
linear operators on $\fnsp{H}$ as $\sL(\fnsp{H})$. We call $\mc{A}^* \in
\sL(\fnsp{H}_2, \fnsp{H}_1)$ the adjoint of a linear operator $\mc{A} \in
\sL(\fnsp{H}_1, \fnsp{H}_2)$ if it holds that $\ipg{\mc{A}u}{v}{\fnsp{H}_2} =
\ipg{u}{\mc{A}^*v}{\fnsp{H}_1}$ for all $u \in \fnsp{H}_1$ and $v \in
\fnsp{H}_2$. 
If $\mc{A} \in \sL(\fnsp{H})$ satisfies $\mc{A} =
\mc{A}^*$, we call $\mc{A}$ self-adjoint. We say that an operator $\mcA
\in \sL(\fnsp{H})$ is positive if $\ipg{x}{\mcA x}{\fnsp{H}} \geq 0$
for all $x \in \fnsp{H}$, and $\mc{A}$ is \emph{strictly} positive, 
if $\ipg{x}{\mcA x}{\fnsp{H}} > 0$ 
for all nonzero $x \in \fnsp{H}$. 
Moreover, a
positive self-adjoint $\mcA \in \sL(\fnsp{H})$ is said to be 
of trace-class if 
$\tr(\mcA) \eqdef \sum_{i=1}^\infty \ipg{\mcA e_i}{e_i}{\fnsp{H}} < \infty$,
where
$\{e_i\}_{i=1}^\infty$ is any orthonormal basis of $\fnsp{H}$. 

In what follows, $\GM{a}{\mcC}$ denotes a Gaussian measure
with mean $a$ and a strictly positive trace-class covariance operator 
$\mcC$. Also,
$\fnsp{U}$, $\fnsp{Z}$, and $\fnsp{M}$ denote the Hilbert spaces
corresponding to the state, control, and inversion parameters, respectively. 

\subsection{Optimal control}
\label{sec:ctn_optimal_control}
We consider optimal control
problems of the form 
\begin{subequations}
\label{eq:OC_problem}
\begin{equation}\label{eq:OC_ctn_obj}
  \min_{\z \in \fnsp{Z}} \Phi(\z; \m) = \frac12 \| \uu(\cdot, T; \m, \z) - \ubr \|_{\fnsp{U}}^2 + \frac\beta2 \| z \|^2_{\fnsp{Z}},
\end{equation}
where the state variable $\uu \in \fnsp{U}$ satisfies
\begin{equation}\label{eq:ctn_state}
\begin{aligned}
   &\frac{d\uu}{dt} = \mcL\uu + \mcC\z + \mc{D}\m + \cc, \\
   &\uu(0) = \uu_0.
\end{aligned}
\end{equation}
\end{subequations}
Here, $\z \in \fnsp{Z}$ is the control variable, $\uu_0 \in \fnsp{U}$ is the
initial state, $\m \in \fnsp{M}$ is an uncertain model parameter, and $\beta > 0$ is
a regularization parameter. 
In~\cref{eq:ctn_state}, $\mc{L}$, $\mc{C}$, and $\mc{D}$ are linear operators, and $\cc \in
\fnsp{U}$ is an affine term.  
Note that subsequent developments can be extended in a straightforward manner to the cases of more general regularization operators.

Let $\uu_T(\m, \z) = \uu(\cdot, T; \m, \z)$
denote the terminal state of $\uu$. This terminal state can be  
obtained via a mapping of the form 
\begin{equation}\label{eq:ctn_control_model}
   \uu_T(m, z) = \mcA \m + \mcB \z + \q, 
\end{equation}
where $\mcA$ and $\mcB$ 
are linear transformations and $\q \in \fnsp{U}$.
We  
assume $\mcA \in \sL(\fnsp{M}, \fnsp{U})$ and $\mcB \in \sL(\fnsp{Z},
\fnsp{U})$. We can thus restate the parameterized
optimal control problem as 
\begin{equation}\label{eq:ctn_OC}
   \min_{\z \in \fnsp{Z}} \Phi(\z; \m) \eqdef \frac12 \| \uu_T(\m, \z) - \ubr \|_{\fnsp{U}}^2 + \frac\beta2 \| z \|^2_\fnsp{Z}.
\end{equation}
For a given $\m$, 
the optimal control $\z^\star_{\m}$ solving~\cref{eq:ctn_OC} satisfies
\[\mcHctrl \ \z^\star_{\m} = \mcB^* (\ubr - \q - \mcA \m),\quad \text{where} \quad \mcHctrl = \mcB^* \mcB + \beta I.\]
Therefore, in the present setting, 
we can express the optimal control as affine 
transformation of $m$, 
$\z^\star_{\m} = - \mcHctrl^{-1} \ \mcB^* \mcA \m + \mcHctrl^{-1} \ \mcB^* (\ubr - \q)$.

\subsection{Bayesian inverse problem}
\label{sec:ctn_inverse_problem}
In the present work, we consider 
inverse problems governed by observation models of the form
\begin{equation}\label{eq:ctn_inversion_model}
   \obs = \mc{F} \m + \vec{b} + \vec\eta.
\end{equation}
Here, $\obs \in \R^{\ny}$ is measurement data, $\mc{F}: \fnsp{M} \to \R^{\ny}$
is a bounded linear transformation, $\vec{b}$ is a constant vector,
and $\vec\eta$ is a random vector modeling measurement noise.
Note that here we have an affine parameter-to-observable map,
$m \mapsto \mc{F} \m + \vec{b}$.

We assume 
$\vec\eta \sim
\GM{\vec{0}}{\noisecov}$, where $\noisecov \in \R^{\ny \times \ny}$ represents
the (symmetric positive definite) noise covariance matrix. 
Furthermore, we assume $m$ and $\vec\eta$ are independent random variables.
We assume a Gaussian prior, 
$\priorm = \GM{m_\mathup{pr}}{\mcC_\mathup{pr}}$, 
for the inversion parameter. Due to the affine structure 
of the parameter-to-observable map and the Gaussian assumption on 
the prior and noise models,  
the solution of the Bayesian inverse problem 
is the Gaussian posterior measure
$\postm = \GM{m_\textup{MAP}}{\mcC_{\mathup{post}}}$ 
where 
\begin{equation}\label{equ:ctn_posterior}
   m_\textup{MAP} = \mcC_{\mathup{post}} \left( \mc{F}^* \noisecov^{-1}(\obs - \vec{b}) + \mcC_{\mathup{pr}}^{-1} m_\mathup{pr} \right)
   \quad\text{and}\quad 
   \mcC_{\mathup{post}} = \left( \mc{F}^* \noisecov^{-1}\mc{F} + \mcC_{\mathup{pr}}^{-1}\right)^{-1}. 
\end{equation}
\subsection{Discretization of control and inverse problems}
\label{sec:disc}
To facilitate computations, the inverse problem and 
the optimal control problem must be discretized.
In this section, we briefly discuss the discretization of 
these problems.

\boldheading{Discretization of Infinite-Dimensional Parameters}
We begin by considering the discretization of the inversion parameter. 
Herein, we assume the inversion parameter space is $L^2(\Omega)$, 
where $\Omega$ is a bounded spatial domain. 
Let $\widehat{\fnsp{M}}$ be a finite-dimensional subspace of $\fnsp{M}$ spanned by
continuous Lagrange nodal basis functions $\{\phi_j\}_{j=1}^{\nm}$. Then, any
$\m \in \widehat{\fnsp{M}}$ can be represented by the vector
$\bs{m} = [\m_1 \; \cdots\; \m_{\nm}]^\top$ of its expansion coefficients. Moreover, note
that for every $m, \ p \in \widehat{\fnsp{M}}$, 
\[
   \ipg{m}{p}{\fnsp{M}} = \sum_{i,j=1}^\nm m_i \ipg{\phi_i}{\phi_j}{\fnsp{M}} p_j = \vec{m}^\top \M \vec{p} 
   \defeq \ip{\vec{m}}{\vec{p}}_\M,
\]
where $\M$ denotes the finite-element mass matrix.
Hence, we represent a discretized parameter $\dpar$ in 
$\R^{\nm}$ equipped with the weighted inner product
$\ip{\cdot}{\cdot}_\M$. 
We denote the inner product space $(\R^{\nm}, \ip{\cdot}{\cdot}_\M)$ by $\R^{\nm}_\M$.
The discretization of the target state $u_T \in \fnsp{U}$ and 
the control variable $z$ are done similarly. See~\Cref{appdx:disc}, 
for details.
We summarize the notations for the discretized 
inversion parameter, and the state and control variables in~\Cref{tab:spaces}.

\begin{table}[ht]
   \centering
   \def\arraystretch{1.5}
   \caption{Discretization of the key quantities in 
            inverse/optimal control problems. Here, 
            $\Rm$, $\Rz$, and $\Ru$ are, respectively, 
            $\R^{\nm}$, $\R^{\nz}$, and $\R^{\nuu}$ equipped with the respective weighted
            inner products and induced norms $\| \cdot \|_\M$, $\| \cdot \|_\Mt$, and $\| \cdot \|_\Mu$.}
   \begin{tabular}{lll}
      \toprule
      \textbf{Quantity} & \textbf{Functional notation} & \textbf{Discretization} \\ 
      \midrule
      inversion parameter & $m \in \fnsp{M}$ & $\bs{m}   \in \Rm$ \\
      control variable & $z \in \fnsp{Z}$    & $\bs{z}   \in \Rz$ \\
      target state & $u_T \in \fnsp{U}$      & $\UT      \in \Ru$ \\ 
      \bottomrule
   \end{tabular}
   \label{tab:spaces}
\end{table}

\boldheading{Discretized Inverse Problem}
The discretized inverse problem seeks to estimate $\vec m$ from the model,
\begin{equation}\label{eq:inversion_model}
   \obs = \FF \dpar + \vec{b} + \vec\eta,
\end{equation}
given a data vector $\obs \in \R^{\ny}$. 
Here, $\FF: \Rm \to \R^{\ny}$ is the discretization of $\mc{F}$
in~\cref{eq:ctn_inversion_model}.
The corresponding posterior distribution for $\vec{m}$ is given by 
$\postm = \GM{\dparmap}{\postcov}$, where 
\begin{equation}
   \label{eq:posterior_exp}
      \dparmap = \postcov \left( \FF^* \noisecov^{-1}(\obs - \vec{b}) + \prcov^{-1}\dparpr \right)
      \quad\text{and}\quad
      \postcov = \left( \FF^* \noisecov^{-1}\FF + \prcov^{-1}\right)^{-1}. 
\end{equation}

\noindent%
Note that here $\prcov$ is the discretized prior covariance operator and $\FF^*$
is the adjoint of $\FF$.  It is important to note that $\FF$ is a linear
transformation from $\R^{\nm}_\M$ to the measurement space, which is equipped
with the Euclidean inner product.  This implies $\FF^* = \M^{-1}\FF^\top$. 
See~\cite{Bui-Thanh2013}
for
further details regarding discretization of infinite-dimensional Bayesian
inverse problems.

\boldheading{Discretized Optimal Control Problem}
We consider
\begin{subequations}
\label{eq:disc_OC_problem}
\begin{equation}
   \min_{\ctrl \in \Rz} \ctrlObjFunc(\ctrl; \dpar) \eqdef \frac12 \| \vec{u}(T) - \ubar \|^2_\Mu 
   + \frac\beta2 \| \vec{z}\|_\Mt^2,
\end{equation}
where the state variable $\vec{u}$ satisfies
\begin{equation}\label{eq:state}
\begin{aligned}
   &\frac{d\vec{u}}{dt} = \mat{L}\vec{u} + \mat{C}\vec{z} + \mat{D}\vec{m} + \vec{c}, \\
   &\vec{u}(0) = \vec{u}_0.
\end{aligned}
\end{equation}
\end{subequations}
Here, $\ctrl$ is the discretized control variable, $\vec{u}_0$ is the discretized
initial condition, and $\dpar$ is the discretized inversion parameter.  Here,
$\mat{L}$, $\mat{C}$ and $\mat{D}$ are (finite-dimensional) linear operators, 
and $\vec{c} \in \Ru$ is an affine term pertaining to the prediction model. 
Note that the terminal state, $\UT \eqdef
\bs{u}(T)$, is obtained via a mapping of the form 
\begin{equation}\label{eq:control_model}
   \UT(\dpar, \ctrl) = \AA \dpar + \BB\ctrl + \vec{q}, 
\end{equation}
where $\AA \in \sL(\Rm, \Ru)$, $\BB \in \sL(\Rz, \Ru)$,
and $\vec{q} \in \Ru$. We also need the 
adjoints of $\AA$ and $\BB$ in what follows. It is simple to 
note that 
$\AA^* = \M^{-1}\AA^{\top}\Mu$ and $\BB^* = \Mz^{-1} \BB^{\top} \Mu$.

Restating 
the parameterized optimal control problem as 
\begin{equation}\label{eq:OC}
  \min_{\vec{z} \in \Rz} \ctrlObjFunc(\ctrl; \dpar) = \frac12 \| \UT(\dpar, \ctrl) - \ubar \|^2_\Mu + \frac\beta2 \| \vec{z}\|_\Mt^2,
\end{equation}
we note that 
the optimal control $\vec{\z}^\star_\dpar$, corresponding to a fixed $\vec m$, 
satisfies
\[
\hessctrl \vec{\z}^\star_\dpar = \BB^* (\ubar - \vec{q} - \AA \dpar),\quad \text{where} \quad \hessctrl = \BB^* \BB + \beta \Mt .
\]
We can therefore express the optimal control $\vec{\z}^\star_\dpar$ as 
\begin{equation}\label{eq:nominal_control}
   \vec{\z}^\star_\dpar \eqdef - \hessctrl^{-1} \ \BB^* \AA \dpar + \hessctrl^{-1} \ \BB^* (\ubar - \vec{q}).
\end{equation}
Recall that the inversion parameter $\dpar$ is unknown and is modeled as a 
random variable. 
A useful estimate of this parameter is given by 
the MAP point $\dparmap$; see \cref{eq:posterior_exp}.
In particular, 
this yields a nominal optimal control  $\ctrln \eqdef \ctrln_{\dparmap}$.
This provides an estimate for the optimal control that
reflects the \textit{best-guess} inversion parameter and is the control that we
would recommend to a practitioner to implement. 

\subsection{Optimal experimental design (OED)}
\label{sec:oed}
We formulate the OED problem  
as that of selecting an optimal subset from a set
$\{\vec{x}_i\}_{i=1}^{\ns}$ of candidate locations for placing sensors.
To facilitate this, we assign a binary weight $w_i \in \{0, 1\}$ to each
candidate location.  Sensors are placed at locations whose corresponding weight
is one.  In this setting, an experimental design is fully specified by a weight
vector $\vec{w} = [w_1 \; \cdots \; w_{\ns}]^\top \in \{0, 1\}^\ns$. 
In what follows, we assume the sensor measurements are uncorrelated and let $\noisecov = \sigma^2 \mat{I}$ for convenience. Note that it is possible to extend the subsequent theory and derivations in a straightforward manner for the case of correlated noise.

Before formulating an OED problem, we need to specify how an experimental 
design is incorporated in the Bayesian inverse problem. 
Following the formulations in~\cite{Alexanderian2013AOptimalDO}, 
we consider the design-dependent posterior distribution 
$\GM{\dparmap(\vec{w}, \obs)}{\postcov(\vec{w})}$, where
\begin{equation}\label{eq:posterior_exp_disc_w}
      \dparmap(\vec{w}, \obs) = \postcov \left( \FF^* \mat{W}_\sigma(\obs - \vec{b}) + \prcov^{-1}\dparpr \right), 
      \quad
      \postcov(\vec{w}) = \left(\FF^* \mat{W}_\sigma\FF + \prcov^{-1}\right)^{-1}, 
\end{equation}
where $\mat{W}_\sigma \eqdef \sigma^{-2}\diag(\vec{w})$. 
Note that whereas $\dparmap$ depends on measurement data, 
the posterior covariance
$\postcov$ depends only on the design $\bs{w}$.
Note also that 
\begin{equation}
   \postcov(\vec w) = \prcov^{1/2} \ (\mat{\tilde H}_{\text{misfit}}(\vec w) + \mat{I})^{-1}\ \prcov^{1/2},
   \label{eq:postcov_hess}
\end{equation}
where $\mat{\tilde H}_{\text{misfit}}(\vec w)$ is the \textit{prior-preconditioned} data
misfit Hessian~\cite{Bui-Thanh2013} defined as, 
\begin{equation}\label{eq:prior_precond_Hm}
   \mat{\tilde H}_{\text{misfit}}(\vec w) = \prcov^{1/2}(\FF^*\mat{W}_\sigma \FF)\prcov^{1/2}.
\end{equation} 
In ill-posed inverse problems with sparse 
data $\mat{\tilde H}_{\text{misfit}}$ can often 
be represented accurately with a low-rank approximation. This 
plays a key role in the computational methods in~\Cref{sec:computational_methods}.

The OED problem seeks to minimize a design criterion 
$\Psi(\bs{w})$ over the set 
$\{ \vec{w} \in \{0, 1\}^{n_s} : \sum_{i=1}^{n_s} w_i = k\}$, where $k$ is a sensor budget.
Note that the design criterion cannot be dependent 
on a specific realization of data, as experimental data is not available \textit{a priori}
at the stage of solving an OED problem. In traditional Bayesian
OED, one typically considers design criteria that quantify 
measures of posterior uncertainty in the inversion parameter 
or information gain. 
An example is the A-optimality criterion 
\begin{equation}\label{eq:aoptimal_criterion}
\Psi^A(\bs{w}) \eqdef \tr\left[ \postcov(\bs{w}) \right],
\end{equation} 
which quantifies the 
average posterior variance
of the inversion parameter.

\section{Control-oriented design criterion and uncertainty quantification}
\label{sec:dataDrivenOC}
In this section, we present our proposed mathematical framework for 
control-oriented OED (cOED) and uncertainty quantification in parameterized
optimal control problems.  We begin by deriving our proposed cOED criterion
in~\Cref{sec:OptimalityCriteria}.
Subsequently, in \Cref{sec:uq}, we characterize the first and second moments of
the optimal control objective and derive a concentration bound for this
quantity.  The latter analysis connects the proposed cOED criterion to the
tail probabilities of the optimal control objective. 

\subsection{A control-oriented design criterion}
\label{sec:OptimalityCriteria}

Let us consider how the posterior uncertainty in the inversion parameter
propagates to the optimal terminal state, i.e., 
the terminal state~\cref{eq:control_model} evaluated at the optimal control.  
Upon parameter inversion, we obtain a
posterior $\postm = \GM{\dparmap^\obs}{\postcov}$ as
defined in \cref{eq:posterior_exp_disc_w}. The resulting MAP parameter estimate,
$\dparmap^\obs$, is then used to compute a nominal optimal control $\ctrln$
\cref{eq:nominal_control}. We apply this control to the system
\cref{eq:control_model}, which results in the optimal terminal state,
\begin{equation}\label{eq:UTs}
   \UTs(\dpar) \eqdef \AA \dpar + \BB \ctrln + \vec{q}.
\end{equation} 

The goal of optimal control is to guide $\UTs(\dpar)$ towards the
target state $\bs{\bar{u}}$. However, the terminal state of the system,
$\UTs(\dpar)$, itself is uncertain due to uncertainty in $\dpar$. 
We can use the posterior distribution $\postm$ of $\dpar$ to obtain a
probabilistic depiction of the terminal state. Specifically, we have
\begin{equation}\label{eq:UTs_law}
   \UTs \sim \GM{\UTs(\dparmap^\obs)}{\mat{A}\postcov\mat{A}^*}.
\end{equation}

In a control-oriented application, it is important to reduce uncertainty in
$\UTs$. 
Following a weighted A-optimal design approach, we define
the cOED criterion as 
\begin{equation}\label{eq:cOED}
    \Psi^{cA} \eqdef \tr(\mat{A} \postcov \mat{A}^*).
\end{equation}
This criterion quantifies the \textit{average} variance in the terminal state 
$\UTs(\dpar)$. Note that this coincides with the average posterior variance of 
$\UTs - \bar{\bs{u}}$. Therefore, the cOED criterion is also representative of
uncertainty in the deviation of the terminal state from the target state. 
It is worth noting that there exist different ways of equivalently expressing
the posterior covariance of $\UTs$ \cite{SpantiniGoalOriented17,WuChenGhattas23}.
Moreover,
one can also view the cOED criterion from a decision-theoretic perspective as
the expected Bayes Risk of $\UTs$:
\begin{equation}
   \Psi^{cA} = \tr(\mat{A} \postcov \mat{A}^*) = \Exp_{\priorm} \left\{ \Exp_{\bs{y} | \bs{m}}\left[ \|\UTs(\dparmap^\obs) - \UTs(\bs{m})\|_{\Mu}^2 \right] \right\},
   \label{eq:cOED_bayes_risk}
\end{equation}
The proof of this relation is given in \cite[Theorem
3.1]{AttiaAlexanderianSaibaba18}. Thus, the cOED criterion also quantifies the
expectation, over the set of all likely data, of the mean-squared deviation of
the optimal terminal state. For further details on  
the concept of Bayes risk, see, e.g.,~\cite[Page 87]{Tenorio17}.

\subsection{Measures of uncertainty in the control objective}
\label{sec:uq}

Upon solving the cOED problem, we obtain a sensor placement. 
Then, we can solve the inverse problem to estimate 
the inversion parameter $\dpar$. Subsequently, the MAP estimate of 
$\dpar$ can be used to solve 
the optimal control problem.  A key downstream consideration is 
characterizing the uncertainty in the optimal control objective.
In particular, we consider the uncertainty in the (non-regularized) control
objective, 
\begin{equation}
   \ctrlObjs(\dpar) \eqdef \frac12 \| \UTs(\dpar) - \ubar\|^2_\Mu = \frac12 \| \AA \dpar + \BB \ctrln +\vec{q} - \ubar\|^2_\Mu.
   \label{eq:ctrlObj}
\end{equation}
Note that $\ctrlObjs$ is a quadratic functional of a Gaussian random variable.
This enables efficient sampling and theoretical analysis of this quantity. In
\Cref{sec:expvar_ctrl_obj}, we  derive the expressions for the expectation and
variance of the control objective and relate these expressions to the cOED
criterion.  Also, in \Cref{sec:conc_bds_ctrlObj}, we derive a
concentration bound that characterizes uncertainty in the control objective in
terms of $\Psi^{cA}$. These results provide a framework 
for quantitatively and qualitatively assessing the uncertainty in the 
control objective.

\subsubsection{Expectation and Variance of $\ctrlObjs$}
\label{sec:expvar_ctrl_obj}
We begin by presenting the expressions for the mean and variance of the
control objective with respect to the posterior law $\postm =
\GM{\dparmap}{\postcov}$ of $\dpar$, whose mean and covariance are defined
in~\cref{eq:posterior_exp}.
\begin{proposition}
\label[prop]{prop:ctrlObjMeanVar}
Let $\Exp(\cdot \, | \, \obs)$ and $\Var(\cdot \,| \, \obs)$ 
denote expectation and variance with respect to 
$\postm$. We have
\begin{subequations}
   \label{eq:ctrlObjMeanVar}
   \begin{align}
      \Exp(\ctrlObjs | \ \obs) &= \frac{1}{2} \tr\left[\AA\postcov\AA^*\right] + \frac{1}{2} \|\AA\dparmap^\obs + \BB \ctrln + \vec{q} - \ubar\|^2_\Mu \label{eq:ctrlObjMean},\\
      \Var(\ctrlObjs| \ \obs) &= \frac{1}{2} \tr\left[ (\AA\postcov\AA^*)^2 \right] + \|\AA\dparmap^\obs + \BB \ctrln + \vec{q} - \ubar \|^2_\mat{\Mu\AA\postcov\AA^*} \label{eq:ctrlObjVar}.
   \end{align}
\end{subequations}
Note that $\ctrln$ depends on $\dparmap$, whose dependence on $\obs$ is highlighted as $\dparmap^{\obs}$ for clarity.
\end{proposition}
\begin{proof} See \Cref{sec:proof_ctrlObjMeanVar}. \end{proof}

Using \Cref{prop:ctrlObjMeanVar}, one can quantitatively estimate the
expected error in the terminal state, having applied the nominal optimal
control corresponding to the MAP point. Note that \cref{eq:ctrlObjMean} can be
restated in terms of the control-oriented design criterion, $\Psi^{cA}$, as
\begin{equation}\label{eq:expAopt}
   \Exp(\ctrlObjs | \ \obs) = \frac{1}{2} \left[ \Psi^{cA} + \ctrlObjs(\dparmap^\obs) 
   \right].
\end{equation}
This suggests that minimization of the control-oriented criterion can
potentially reduce the expected control objective. However, due to the
dependence of \cref{eq:expAopt} on data, minimization of $\Psi^{cA}$ may not
guarantee minimization of $\Exp(\ctrlObjs | \ \obs)$. Nevertheless, in the case
where the variance of $\Psi^{cA}$ over designs dominates
$\ctrlObjs(\dparmap^\obs)$, the expectation and variance of the control
objective is often minimized with a cOED. 

We also note that the cOED criterion can be bounded by the classical 
A-optimality criterion~\cref{eq:aoptimal_criterion}. Namely, 
it is straightforward to show that 
\[
   \frac{|\Psi^{cA}-\Psi^{A}|}{\Psi^{A}} \leq \| \mat{A}^* \mat{A} - \mat{I}\|,
\]
where the norm in the right-hand side is the operator norm induced by the norm 
$\| \cdot \|_\M$ on $\Rm$; see \cite[Fact~5.12.7]{BernsteinMatrixMath11}. 
As a result, in situations where 
$\mat{A}$ has near-orthogonal columns, minimizing the classical A-optimality 
criterion can be beneficial in terms of reducing 
the cOED criterion. This, however, does not hold in general, where 
$\| \mat{A}^*\mat{A} - \mat{I}\|$ might be large.  
\subsubsection{Concentration Bounds for $\ctrlObjs$}
\label{sec:conc_bds_ctrlObj}
In practical settings, it is often of interest to gauge the probability of a
\textit{worst-case} scenario, i.e., the case where the nominal control yields an
unexpectedly large control objective.  Here, we analyze the probability
of deviations of the control objective from its expectation and
demonstrate the advantages of a control-oriented design in this aspect. 
The following result \Cref{prop:ctrlObjConcentration} and the subsequent 
corollary provide insight into this.

\begin{proposition}
   [Concentration bound for $\ctrlObjs$.]%
   \label[prop]{prop:ctrlObjConcentration}
   For every $\tau \geq 0$, 
   \begin{equation}
      \P\left( \left| \ctrlObjs(\dpar) -  \Exp(\ctrlObjs | \ \obs)\right| \geq \tau\right) \leq 4\exp\left[ - \frac{1}{8}\min\left\{ \frac{\tau}{\Psi^{cA}}, \frac{\tau^2}{(\Psi^{cA})^2}, \frac{\tau^2}{C^2} \right\} \right],
      \label{eq:ctrlObjConcentration}
   \end{equation}
   where
   $C = \|\AA\dparmap + \BB \ctrln + \vec{q} - \ubar\|_\mat{\M_u\AA\postcov\AA^*}$.
\end{proposition}
\begin{proof} See \Cref{sec:proof_ctrlObjConcentration}. \end{proof}
\begin{corollary}
   Consider the setup from \Cref{prop:ctrlObjConcentration}. Let $0 < \delta < 1$. Then, with probability of at least $1-\delta$, 
   \begin{equation}
      \left| \ctrlObjs(\dpar) -  \Exp(\ctrlObjs | \ \obs)\right| \leq \sqrt{8\log(4/\delta)}\max\left\{\sqrt{8\log(4/\delta)}\Psi^{cA}, \Psi^{cA}, C \right\},
      \label{eq:ctrlObjConcentrationInterval}
   \end{equation}
   where $C$ is as defined in \Cref{prop:ctrlObjConcentration}.
   \label[cor]{cor:controlObjInterval}
\end{corollary}
\begin{proof} See \Cref{sec:proof_ctrlObjConcentrationCorollary}. \end{proof}
Note that for large deviations $\tau$, the concentration bound
\cref{eq:ctrlObjConcentration} behaves like $\exp(-\tau/\Psi^{cA})$. In other
words, the probability of the control objective being much larger than its
expectation (i.e., probability of a worst-case scenario) is restricted by
$\Psi^{cA}$. Similarly, \Cref{cor:controlObjInterval} shows that the size of
high-probability confidence regions for the control objective is proportional to
$\Psi^{cA}$.  Based on these observations, it is expected that an experimental
design minimizing the control-oriented criterion $\Psi^{cA}$ would aid in
reducing uncertainty in the control objective. Computational results in
\Cref{sec:num_res_oed_uq} provide further intuition for the established
relationship between the expectation and variance of $\ctrlObjs$ and the cOED
criterion.

\section{Computational methods}
\label{sec:computational_methods}
In this section, we discuss the computational methods used for solving the
control-oriented OED (cOED) problem and quantifying the uncertainty in the
control objective.  We rely on a greedy approach for approximately solving the
cOED problem; see \Cref{sec:greedy}. Such optimization routines involve
repeated evaluations of the cOED criterion. Hence, efficient estimation of the
cOED criterion is essential. In \Cref{sec:computeAopt}, we outline fast
computational methods for computing this criterion. Then, in
\Cref{sec:ctrlObjEstimation} we discuss fast computation of measures of
uncertainty in the control objective.  

\subsection{Greedy sensor placement}
\label{sec:greedy}
To find an optimal sensor placement, we need to minimize the cOED 
criterion 
given
a fixed budget $k$ of sensors. Solving this combinatorial optimization problem 
via exhaustive search is
computationally intractable in practice. 
A practical approach for approximately solving this problem is to use a
greedy procedure; see~\Cref{alg:greedy}.  In this approach, sensors
are selected by successively identifying a candidate sensor location in each
stage that locally minimizes the design criterion.  Note that the minimizer in
line 3 of \Cref{alg:greedy} may not be unique, in which case, we
arbitrarily select one of the minimizers.

Despite being suboptimal, the greedy
approach can provide useful designs in practice, and its effectiveness has been analyzed rigorously in the context of experimental design
\cite{Jagalur2021Batch,Bian2017Guarantees}. In place of the greedy
algorithm, one may also consider exchange algorithms \cite{Harman2018ARE},
relaxation approaches \cite{Alexanderian2013AOptimalDO}, stochastic binary
optimization \cite{Attia2021StochasticLA}, or other heuristic methods for
optimization \cite{Huan2024OptimalED}. Note that \Cref{alg:greedy} still
requires many evaluations of the criterion, roughly in the order of $k\ns$.
Hence, efficient computation of the design criterion is critical to the
feasibility of the method, which we address next. 

\begin{algorithm}[ht]
   \caption{Greedy Sensor Placement}
   \begin{algorithmic}[1]
   \Require Target number of sensors $k \leq n_s$ and design criterion $\Psi: \vec{w} \mapsto \Psi(\vec{w}) \in \R^+$ 
   \Ensure Design vector $\vec{w} \in \R^{\ns}$
   \State Set $\vec{w} = \vec{0}$, $\mathcal{U} = \{1,\ldots,\ns\}$, and $\mc{S} = \emptyset$
   \For{$i = 1, 2, \ldots, k$}
       \State $i^* = \argmin_{j \in \mc{U} \setminus \mc{S}} \Psi(\vec{w} + \vec{e}_j)$
       \State $\mc{S} = \mc{S} \cup \{i^*\}$
       \State $\vec{w} = \vec{w} + \vec{e}_{i^*}$
   \EndFor
   \end{algorithmic}
   \label{alg:greedy}
\end{algorithm}

\subsection{Efficient estimation of cOED criterion}
\label{sec:computeAopt}
Consider the control-oriented A-optimal criterion from \cref{eq:cOED}:
\begin{equation}
   \Psi^{cA} = \tr(\AA   \postcov \AA^*) = \sum_{i=1}^{\nuu} \ipg{\AA \postcov \AA^*\vec{e}_i}{\vec{e}_i}{\Mu},
   \label{eq:traceAoptimal}
\end{equation}
where $\{\vec{e}_i\}_{i=1}^\nuu$ is an orthonormal basis of $\Ru$.
Direct computation of \cref{eq:traceAoptimal} requires $\nuu$ matrix vector
products with $\postcov$.  This is prohibitively expensive in practice---the
discretized parameter dimension is typically on the order of thousands or tens
of thousands.  Moreover, computing an application of $\postcov$ using an iterative method (e.g., CG or an appropriate Krylov subspace method) is expensive due to the need for repeated applications of $\FF$
and $\FF^*$, i.e., forward and adjoint PDE solves. Note also that evaluating
the trace in~\cref{eq:traceAoptimal} requires $\nuu$ matrix-vector products with
the adjoint of the goal-operator $\AA$; i.e., $\nuu$ additional PDE solves. 
This section is about addressing these challenges. 

We first discuss fast methods for 
computing applications of 
$\postcov$ on vectors. Two approaches are considered. 
The first one, which we refer to as the \emph{spectral approach}, 
relies on a low-rank spectral decomposition of the prior-preconditioned 
data-misfit Hessian. The second approach, which we call 
the \emph{frozen low-rank forward operator approach}, 
uses a low-rank SVD of the prior-preconditioned forward 
operators. As we will see, the latter approach has the benefit of being design-independent and therefore is more suitable for our purposes.
\subsubsection{Fast $\postcov$ apply using the spectral approach}
\label{sec:hessian_lora}
As noted previously, the prior-preconditioned Hessian~\cref{eq:prior_precond_Hm}
often can be represented accurately using a low-rank spectral 
decomposition:
\begin{equation}\label{eq:lowrank_Hm}
\mat{\tilde H}_{\text{misfit}} \approx \mat{V}_h \mat{\Lambda}\mat{V}_h ^*.
\end{equation} 
Here, $\mat{\Lambda} \in \R^{{\kh} \times {\kh}}$ is a diagonal
matrix with the leading eigenvalues of $\mat{\tilde H}_{\text{misfit}}$ on its diagonal and 
$\mat{V}_h  \in \sL(\R^{{\kh}}, \Rm)$ satisfies $\mat{V}_h^* \mat{V}_h = \mat{I}$,
where 
$\mat{V}_h ^* = \mat{V}_h ^\top \mat{M}$.
This decomposition can be
obtained iteratively using  
the Lanczos method \cite{Lanczos1950} or through
approximation techniques such as a randomized Nystr\"om decomposition \cite{Gittens2013RevisitingTN}. This
generally requires at least ${\kh}$ applications of $\mb{F}$ and its
adjoint $\mb{F}^*$. Note that the data-misfit Hessian \cref{eq:lowrank_Hm} and its low-rank approximation depend on the design vector $\vec{w}$ by \cref{eq:prior_precond_Hm}. Moreover, by \cref{eq:prior_precond_Hm}, it follows that the rank of the data-misfit Hessian does not exceed $n_s$.

Once the low-rank approximation~\cref{eq:lowrank_Hm} is available, 
we can use the Sherman-Morrison-Woodbury identity
\cite{Hager1989UpdatingTI} to obtain
\[
(\hessmis + \mat{I})^{-1} \approx (\mat{V}_h \mat{\Lambda}\mat{V}_h ^* + \mat{I})^{-1} = \mat{I} - \mat{V}_h \mat{D}\mat{V}_h ^*,
\] 
where $\mat{D} \eqdef \diag(\lambda_1/(\lambda_1 +
1), \lambda_2/(\lambda_2 + 1), \dots, \lambda_{\kh}/(\lambda_{\kh} + 1))$. This
yields the following approximation for the posterior covariance,
\begin{equation}\label{eq:postcov_spectral}
   \postcov\approx \postcovh \eqdef \prcov - \prcov^{1 / 2}\ \mat{V}_h \mat{D}\mat{V}_h ^* \ \prcov^{1 / 2}.
\end{equation}
This approximation of $\postcov$ has been used in various previous 
studies in Bayesian analysis and OED; see, e.g.,~\cite{Bui-Thanh2013,Alexanderian2017EfficientDD}.

\subsubsection{Fast $\postcov$ apply using the frozen 
low-rank forward operator approach}
\label{sec:frozen_lora}
Consider the $\vec{w}$-dependent prior-preconditioned data-misfit Hessian
\[
\hessmis(\vec w) = \prcov^{1/2}\FF^* \mat{W}_\sigma \FF\prcov^{1/2} = \tilde{\FF}^* \mat{W}_\sigma \tilde{\FF}.
\]
Here, $\tilde{\FF} \eqdef \FF\prcov^{1/2}$ is the 
prior-preconditioned forward operator. We utilize the generalized randomized SVD algorithm
\cite{Halko2009FindingSW,Saibaba2021GeneralizedSVD} with target rank $\kf$ and modest oversampling parameter ($p = 5$) to compute $\tilde{\FF}
\approx \mb{U}_F \mb{V}_F^*$, where $\mb{U}_F \in \R^{\ny \times \kf}$ is a
matrix with orthonormal columns, and $\mb{V}_F \in \sL(\R^{\kf}, \Rm)$ encodes
the right singular vectors and singular values. Substituting this
approximation into \cref{eq:postcov_hess} yields
\[
   \postcov\approx \postcovf \eqdef \prcov^{1/2}\left( \mat{V}_F\mat{U}_F^\top\mat{W}_\sigma \mat{U}_F \mat{V}_F^* + \mat{I}\right)^{-1}\prcov^{1/2}.
\]
For notational convenience, we define $\mat{C}_w  \in \R^{k_f \times k_f}$ by 
\begin{equation}
   \mat{C}_w \eqdef \mat{U}_F^\top\mat{W}_\sigma \mat{U}_F, 
\end{equation}
and use 
$\postcovf = \prcov^{1/2}\left( \mat{V}_F \mat{C}_w \mat{V}_F^* + \mat{I}\right)^{-1}\prcov^{1/2}$.
Using the Sherman-Morrison-Woodbury identity, we have 
$\left( \mat{V}_F \mat{C}_w \mat{V}_F^* + \mat{I}\right)^{-1} = \mat{I} - \mat{V}_F\left( \mat{C}_w\mat{V}_F^* \mat{V}_F + \mat{I} \right)^{-1} \mat{C}_w\mat{V}_F^*$.
Therefore,
\begin{equation}
   \postcovf = \prcov - \prcov^{1/2}\big[ \mat{V}_F\left( \mat{C}_w\mat{V}_F^* \mat{V}_F + \mat{I} \right)^{-1} \mat{C}_w\mat{V}_F^* \big]\prcov^{1/2}.
   \label{eq:frozen_postcov}
\end{equation}
Note that as opposed to \cref{eq:postcov_hess}, which involves the inverse of an
operator in $\R^{\nm \times \nm}$, \cref{eq:frozen_postcov} only requires the
inverse of a matrix in $\R^{\kf \times \kf}$.  In practice, typically, $\kf$ is
much smaller than the discretized parameter dimension $N$.  Moreover, once the
discretization of the problem is sufficiently fine to resolve the dominant
singular values of $\tilde\FF$, the numerical rank $k_f$ does not grow upon
successive mesh refinements.

\subsubsection{Fast computation of the cOED criterion}
\label{sec:fast_cOED}
The frozen low-rank forward operator approach discussed above is particularly
well-suited to the present Bayesian inverse problem setup. Thus, we first present 
a procedure for 
fast computation of the cOED criterion using this
approach.  As seen shortly, at the up-front cost of precomputing the low-rank
approximation of $\tilde{\FF}$, we can solve the cOED optimization problem with
no need for $\FF$ and $\FF^*$ applies, which require forward and adjoint PDE solves. Moreover, by precomputing $k_f$
applications of the goal operator $\AA$, the cOED criterion
can be estimated 
with no additional solves of the PDEs governing the optimal control problem or the 
corresponding adjoint PDEs. It is noteworthy that access to the adjoint of the goal operator, $\AA^*$, is not required in this approach. 

Using the low-rank SVD of $\tilde\FF$, we obtain the approximation
\[
\Psi^{cA} = \tr(\AA   \postcov \AA^*) \approx \tr(\AA \  \postcovf \ \AA^*) =: \widetilde{\Psi}^{cA}_{\mathrm{f}}.\]
Substituting the expression from \cref{eq:frozen_postcov}, we have 
\begin{align*}
   \AA \  \postcovf \ \AA^* &= \AA \left( \prcov - \prcov^{1/2}\left[ \mb{V}_F\left( \mb{C}_w\mb{V}_F^* \mb{V}_F + \mb{I} \right)^{-1} \mb{C}_w\mb{V}_F^* \right]\prcov^{1/2} \right) \AA^* \\
   &= \AA\prcov\AA^* - \widetilde{\mat{A}}\left( \mb{C}_w\mb{V}_F^* \mb{V}_F + \mb{I} \right)^{-1} \mb{C}_w \widetilde{\mat{A}}^*,
\end{align*}
where $\widetilde{\mat{A}} \eqdef \AA \prcov^{1/2}\mb{V}_F$. 
A straightforward algebraic manipulation yields
\begin{equation}
   \widetilde{\Psi}^{cA}_{\mathrm{f}} = \tr(\AA \  \postcovf \ \AA^*) = \tr(\AA\prcov\AA^*) - \tr\left[ \left( \mb{C}_w\mb{V}_F^* \mb{V}_F + \mb{I} \right)^{-1} \mb{C}_w \widetilde{\mat{A}}^* \widetilde{\mat{A}}\right].
   \label{eq:frozen_psi}
\end{equation}
Recall that the goal of cOED problem is to minimize $\widetilde{\Psi}^{cA}_{\mathrm{f}}$ over the design $\bs{w}$. Since $\tr(\AA\prcov\AA^*)$ does 
not depend on $\vec w$, we can instead focus on minimizing 
\begin{equation}
   \widetilde{\Psi}^{cA}_{\mathrm{f-}} \eqdef - \tr\left[ \left( \mb{C}_w\mb{V}_F^* \mb{V}_F + \mb{I} \right)^{-1} \mb{C}_w \widetilde{\mat{A}}^* \widetilde{\mat{A}}\right].
   \label{eq:trMaxCrit}
\end{equation}
Note that we have expressed the design criterion
$\widetilde{\Psi}^{cA}_{\mathrm{f-}}$ as the trace of a $k_f$-dimensional
operator. The expression in \cref{eq:trMaxCrit} involving the trace can be computed efficiently with no further need 
for applications of $\FF$ and its adjoint.  Furthermore, we may also precompute
the operator $\widetilde{\mat{A}} = \AA \prcov^{1/2}\mb{V}_F$ by performing
$k_f$ applications of the goal operator $\AA$.
With the requisite pre-computations in place, all subsequent 
evaluations of $\widetilde{\Psi}^{cA}_{\mathrm{f-}}$ can be performed 
with no need for further PDE solves. 
We summarize the steps for computing $\widetilde{\Psi}^{cA}_{\mathrm{f-}}$
in~\Cref{proc:frozen}.

\begin{procedure}{Frozen Low-Rank Evaluation of cOED Objective}
   \label{proc:frozen}
   \noindent\textit{Offline}
   \begin{description}
      \item Approximate $\tilde{\FF} \approx \mb{U}_F \mb{V}_F^*$ 
      \hfill\Comment{cost: $\mc{O}(k_f)$ applies of $\FF$ and $\FF^*$}
      \item Precompute $\widetilde{\mat{A}} = \AA \prcov^{1/2}\mb{V}_F$ 
      \hfill\Comment{cost: $k_f$ applications of goal operator $\AA$}
      \item Compute $\widetilde{\mat{A}}^*\widetilde{\mat{A}}$ and $\mb{V}_F^* \mb{V}_F$
   \end{description}

   \noindent\textit{Online}
   \begin{description}
      \item Compute $\mb{C}_w = \mb{U}_F^\top\mb{W}_\sigma \mb{U}_F$ and evaluate $\cref{eq:trMaxCrit}$ 
   \end{description}
\end{procedure}

Thus far, we have focused on computing 
the cOED criterion using the frozen low-rank forward operator 
approach. 
The spectral approach discussed in \Cref{sec:hessian_lora} can be  
used to derive an approximate criterion analogous to
\cref{eq:trMaxCrit}. One important downside to 
the spectral approach is that the low-rank
approximation of the prior-preconditioned data-misfit Hessian 
depends on the design $\vec w$. Therefore, for each evaluation of the
cOED objective, $\mc{O}(k_h)$ forward/adjoint PDE solves for the Hessian 
approximation are necessary. Note, however, that $k_h$ is typically 
much smaller than the numerical rank $k_f$ of $\tilde{\FF}$, 
which was considered in the frozen low-rank approach. Namely, 
$k_h$ is bounded by the number of active sensors in a 
design; i.e., for a given design $\vec w$, $k_h \leq \| \vec{w}\|_1$. 
Given $\vec w$, we can compute $\widetilde{\mat{A}}_h = \AA
\prcov^{1/2}\mb{V}$ to obtain the approximation, 
\[\mat{A}\postcovh\mat{A}^* = \mat{A}\prcov\mat{A}^* - \widetilde{\mat{A}}_h\mat{D}\widetilde{\mat{A}}_h^*.\] 
As before, we can ignore the design-invariant terms when optimizing 
the cOED criterion---we can instead focus on 
minimizing 
\[
\widetilde{\Psi}^{cA}_{\mathrm{h-}}(\vec w) \eqdef 
-\tr(\widetilde{\mat{A}}_h\mat{D}(\vec w)\widetilde{\mat{A}}_h^*).
\] 
We summarize the cost of the two approaches when using greedy optimization in terms of the number of PDE solves in \Cref{tab:comp_cost}. Note that we exclude the cost associated with applying prior covariance operators, as they are often much smaller than the cost of applying $\AA$, $\FF$, and their adjoints. It is notable, nevertheless, that the frozen approach does not require any online applications of the prior covariance operator either.
\begin{table}[h]
   \centering
   \begin{tabular}{c|c|c}
      \textbf{Method} & \textbf{Total Offline Cost} & \textbf{Total Online Cost} \\\hline 
      Spectral approach (\Cref{sec:hessian_lora}) & --- & $\mathcal{O}(k_h \cdot kn_s)$ \\
      Frozen approach (\Cref{sec:frozen_lora}) & $\mathcal{O}(k_f)$ & --- \\
   \end{tabular}
   \caption{Computational cost of the methods (in number of PDE solves)}
   \label{tab:comp_cost}
\end{table}

To evaluate the entire expression for the 
cOED criterion $\Psi^{cA}$, with either approach, we need the
design-invariant quantity $\tr(\AA\prcov\AA^*)$. Although not necessary for
optimization, efficient matrix-free evaluation of this quantity is needed
for post-inference analysis (see \Cref{sec:ctrlObjEstimation}). 
A similarity transform of $\AA\prcov\AA^*$ yields the symmetric,
positive-definite matrix $\Mu^{1 / 2}\AA\prcov\AA^*\Mu^{-1 / 2}$, 
which is convenient to work with. Since the trace
is invariant under a similarity transformation, we 
focus on computing
$\tr(\Mu^{1 / 2}\AA\prcov\AA^*\Mu^{-1 / 2})$.
In this article, we employ the structure-exploiting $\alg{XNysTrace}$ algorithm
proposed in \cite{Epperly2023XTraceMT} for the efficient estimation of this
trace. Since this quantity is design-invariant, it only needs to be computed once. 
\subsection{Efficient estimation of expectation and variance of $\ctrlObjs$}
\label{sec:ctrlObjEstimation} 
Here, we discuss fast computation of the posterior mean 
and variance of the control objective. Let us 
first consider $\Exp(\ctrlObjs | \obs)$, which can be represented 
as~\cref{eq:expAopt}.
Note that computing this expected value 
requires the computation of the criterion $\Psi^{cA}$, the
MAP estimate $\dparmap^\obs$, and the objective $\ctrlObjs(\dparmap^\obs)$. Of
these quantities, the computational cost of estimating $\Exp(\ctrlObjs | \
\obs)$ is dominated by the estimation of the cOED criterion, for which efficient
techniques are outlined in \Cref{sec:computeAopt}. 

We next consider the expression for
the variance of the control objective from \cref{eq:ctrlObjMeanVar}.
The key challenge for computing this variance expression 
is computing the trace of $(\AA\postcov\AA^*)^2$.
This trace can be estimated efficiently by
utilizing the approximate posterior covariance $\postcovf$. 
In particular, note
that 
\begin{align*} 
   \tr\big( (\AA\postcovf\AA^*)^2 \big) 
   &= \tr\Big( \big(
      \AA\prcov\AA^* - \widetilde{\mat{A}}\left( \mb{C}_w\mb{V}_F^* \mb{V}_F + \mb{I}
      \right)^{-1} \mb{C}_w \widetilde{\mat{A}}^* \big)^2 \Big) \\ 
   &= 
   \tr\big((\AA\prcov\AA^*)^2 \big) + 
      \tr\Big( \big((
      \mb{C}_w\mb{V}_F^* \mb{V}_F + \mb{I})^{-1} \mb{C}_w
      \widetilde{\mat{A}}^*\widetilde{\mat{A}} \big)^2 \Big) \\ 
   &\qquad - 2\tr\big((\mb{C}_w\mb{V}_F^* \mb{V}_F + \mb{I})^{-1} \mb{C}_w
   \mat{G}^* \mat{G}\big), %
\end{align*} 
where $\mat{G} \eqdef \prcov^{1 / 2}\mat{A}^*\widetilde{\mat{A}}$. We
compute the design-independent additive quantities in the above expression in an
offline efficient matrix-free manner similar to as described in the end of
\Cref{sec:computeAopt}. In fact, computations for estimating
$\tr(\AA\prcov\AA^*)$ can be reused for estimating $\tr\big((\AA\prcov\AA^*)^2\big)$. Specifically, we can reuse the randomized Nystr\"om approximation of $\AA\prcov\AA^*$ computed in \alg{XNysTrace} to obtain an approximation of $(\AA\prcov\AA^*)^2$. The trace of this approximation can be used as an estimate to $\tr\big((\AA\prcov\AA^*)^2\big)$, avoiding additional PDE solves.

\section{Numerical results}
\label{sec:num_res}
In this section, we present our numerical experiments pertaining to the
cOED framework. 
The motivating example of a two-dimensional heat transfer application
from \Cref{sec:motivation} is used as a testbed for the proposed approach. 
In
\Cref{sec:num_res_ip,sec:num_res_ocp}, we describe the discretized inverse and
optimal control problems for this model problem. 
We next study the low-rank approximation of the prior-preconditioned 
parameter-to-observable map and its use for fast computation of 
the cOED criterion in~\Cref{sec:num_valid}.
This analysis lays the groundwork for \Cref{sec:num_res_oed}, in which we solve the OED
problem and demonstrate the effectiveness of cOED in comparison with classical
approaches.
\subsection{Inverse problem} \label{sec:num_res_ip} Consider the
problem of reconstructing the unknown source parameter $m$ 
from sensor measurements of $\invState$ in~\cref{eq:noControlHeat}. To discretize the problem, we employ linear triangular continuous
Galerkin finite elements with $\nm = \nuu = 961$ degrees of freedom (a $30
\times 30$ mesh), yielding discretized parameters $\dpar$ and $\bs{\hat{u}}$.
Following the approach in \cite{Bui-Thanh2013}, we define a prior distribution
$\priorm = \GM{\bs{0}}{\prcov}$, where the prior covariance operator is the
discretization of a squared inverse of an elliptic differential operator. In
particular, we define $\prcov^{1 / 2} \eqdef \left( \alpha\mat{K} + \beta\M
\right)^{-1}\M$, where $\M$ and $\mat{K}$ represent the finite-element mass and
stiffness matrices, respectively. Here, we use $\alpha = 0.1$
and $\beta = 1$. To provide some insight, 
we show the prior variance field in \Cref{fig:num_res_map_priorvar}. 

\begin{figure}[ht]
   \centering
   \includegraphics[width=0.45\textwidth]{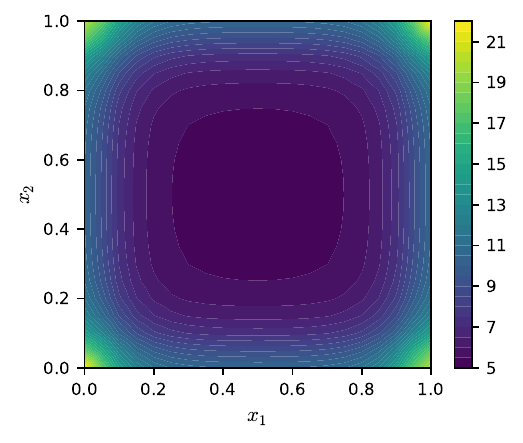}
\caption{Prior pointwise variance field}
\label{fig:num_res_map_priorvar}
\end{figure}

To facilitate the numerical experiments that follow, we generate synthetic data
by  manufacturing a \textit{ground-truth} inversion parameter,
$\bs{m}_{\text{true}}$, as shown in \Cref{fig:num_res_one} (left). Using
$\dpartrue$, we solve the forward model \cref{eq:noControlHeat} to obtain
sensor data $\bs{y}$ on a uniform grid of $\ny = 81$ candidate sensor
locations, as visualized in \Cref{fig:num_res_one}~(right). 
To emulate
measurement noise, we add to this $\obs$ a random draw 
$\bs{\eta} \sim
\GM{\bs{0}}{\noisecov}$. Here, we use $\noisecov = \sigma^2 \mb{I}_{\ny}$,
where $\sigma \eqdef \frac{\delta}{\sqrt{\ny}} \|\mb{F}\bs{m}_{\text{true}} +
\bs{b}\|_2$ with $\delta = 0.01$, resulting in a (normalized) noise level of $1\%$. 

\begin{figure}[ht]
      \centering
      \includegraphics[width=0.4\textwidth]{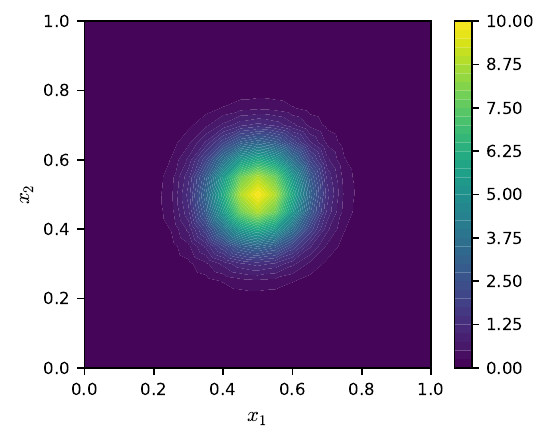}
      \includegraphics[width=0.4\textwidth]{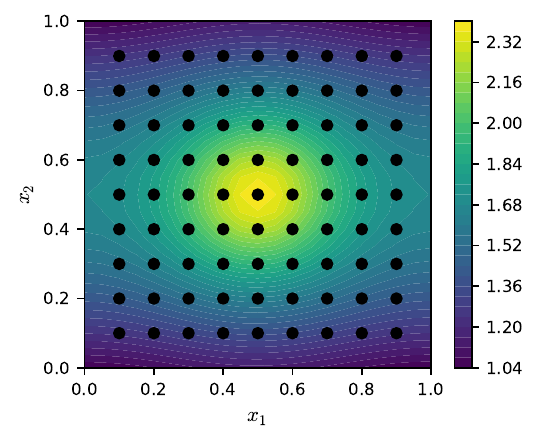}
   \caption{Ground-truth inversion parameter $\dpartrue$ (left) and initial condition $\vec{\invState}$ (right).}
   \label{fig:num_res_one}
\end{figure}

We solve the inverse problem using synthetic measurement data obtained at all candidate sensor locations, and in \Cref{fig:num_res_map_postvar}, we plot the resulting posterior mean $\dparmap$ and the pointwise posterior variance field. 

\begin{figure}[ht]
   \centering
   \includegraphics[width=0.45\textwidth]{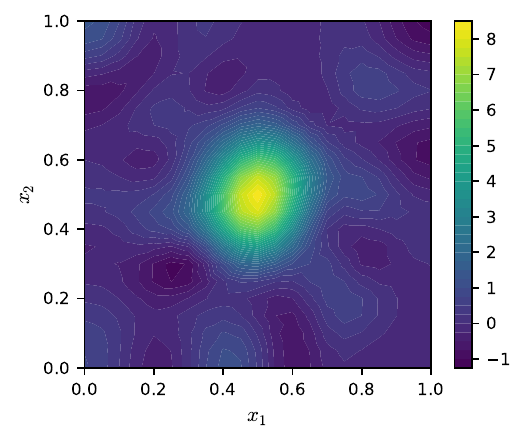}
   \quad
   \includegraphics[width=0.45\textwidth]{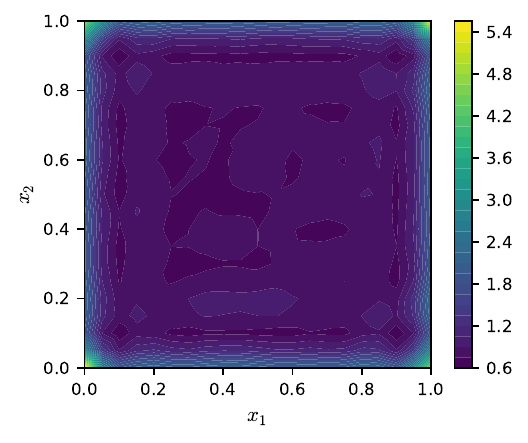}
\caption{Posterior mean $\dparmap$ (left) and pointwise variance (right).}
\label{fig:num_res_map_postvar}
\end{figure}

\subsection{The optimal control problem}
\label{sec:num_res_ocp}
Let us revisit the optimal control problem pertaining to the motivating
application in~\Cref{sec:motivation}. We would like to solve for an optimal control that guides the
transient advection-diffusion model towards a target state by controlling the intensity of a heat source. 
Here, we consider a uniform target state, $\bar{\bs{u}} \equiv
1$  and the velocity field $\vec{v}(\vec x) \eqdef [ -x_2 - 0.5 \; x_1 -0.5]^\top$. This solenoidal velocity field corresponds to a counterclockwise circular field around the center, $(x_1, x_2) = (0.5, 0.5)$.
We consider a time interval $[0, 1]$ discretized by $\nt = 20$
evenly spaced time steps. The forward and adjoint PDEs are solved using backward
Euler time stepping. In addition, to integrate spatio-temporal
quantities, we utilize a temporal mass matrix $\Mt$ based on the composite 
midpoint rule. 

Given this setup, we solve the optimization problem \cref{eq:OC} with regularization parameter $\beta \eqdef 10^{-5}$ to obtain a nominal control estimate, $\bs{z}^*$, corresponding to $\dparmap$. The resulting control and terminal state are shown in \Cref{fig:num_res_nomCtrl_uT}. The nominal control yields a (non-regularized) control objective of $\ctrlObjs(\dpartrue) = 0.039$, yielding a terminal state that is 83\% closer to the target state compared to the initial state.
\begin{figure}[ht]
   \centering
   \includegraphics[width=0.48\textwidth]{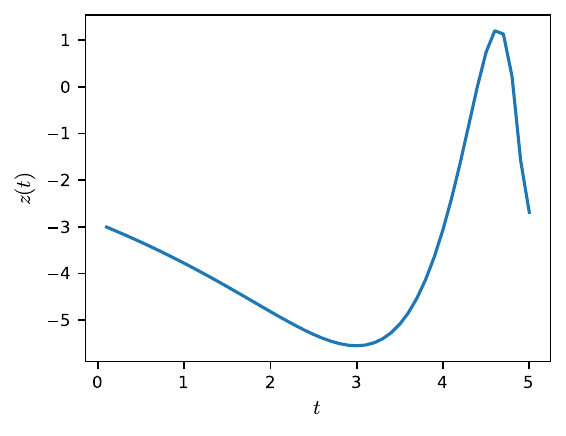}
   \quad
   \includegraphics[width=0.45\textwidth]{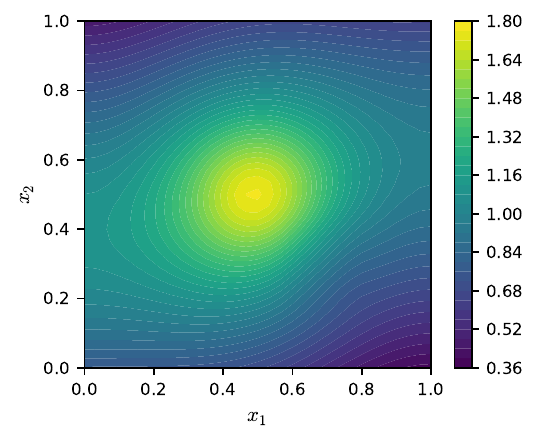}
\caption{Nominal control $\ctrln$ (left) and terminal state $\UTs(\dparmap, \ctrln)$ (right).}
\label{fig:num_res_nomCtrl_uT}
\end{figure}

\subsection{Spectral analysis of prior-preconditioned forward operator}
\label{sec:num_valid}
In this section, we study the spectral decay of the prior-preconditioned forward operator. This provides insight regarding the accuracy of the frozen low-rank approximation technique in the computation of the cOED criterion and the uncertainty measures in \Cref{sec:computational_methods}.
\begin{figure}[ht]
   \centering
   \includegraphics[width=0.45\textwidth]{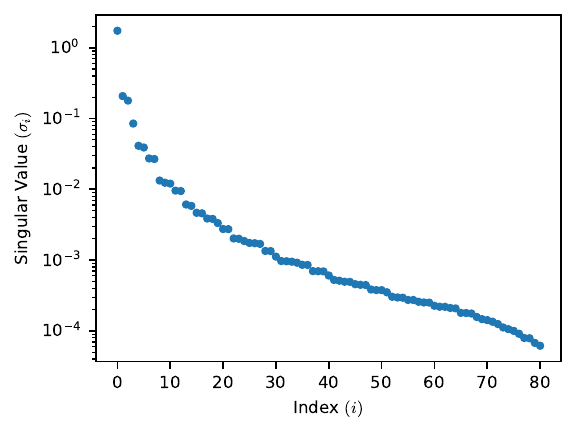}
   \quad
   \includegraphics[width=0.45\textwidth]{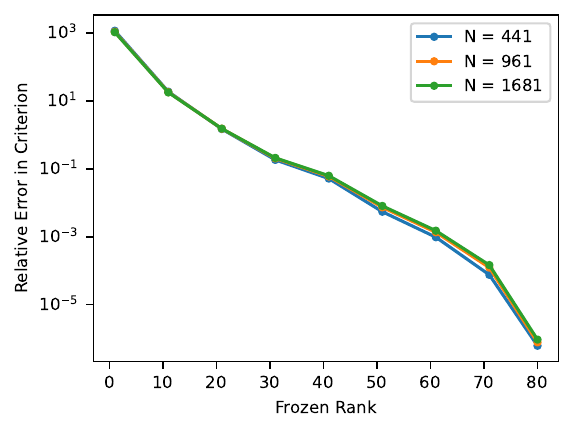}
\caption{Singular values of $\tilde\FF$ (left) and relative error of $\Psi^{cA}$ over $\kf$ (right).}
\label{fig:num_res_fh_spectral}
\end{figure}

\Cref{fig:num_res_fh_spectral} (left) shows rapid decay of the singular values
corresponding to the prior-preconditioned forward operator. Recall that the rank
of $\tilde\FF$ is limited by the dimension of candidate sensors, which is
already typically much smaller than the discretized parameter dimension.
Furthermore, the spectral decay of this operator observed here validates that an
even lower-rank approximation can capture the dominant features of the operator. 

In addition to visualizing the spectral decay of $\tilde\FF$, we analyze the
impact of rank-$\kf$ truncation on the accuracy of the estimated cOED criterion
in~\Cref{fig:num_res_fh_spectral} (right). This is analyzed by computing the relative error of the estimated cOED criterion using~\Cref{proc:frozen} with the cOED criterion computing the exact posterior. We observe a near-exponential decay
of the error, which reinforces that modest low-rank approximations can produce
accurate estimates of the criterion. In \Cref{fig:num_res_fh_spectral} (right),
we also plot the relative error of $\Psi^{cA}$ across different mesh resolutions.  The results show
consistent criterion estimates across all resolutions. 

\subsection{Optimal experimental design}
\label{sec:num_res_oed}
In this section, we investigate classical and cOED
techniques and demonstrate the effectiveness of the control-oriented approach. 

\subsubsection{Solving the OED Problem}
\label{sec:num_res_oed_solve}
We consider a fixed budget of $k = 13$ sensors.  In the experiments that follow,
we use the greedy approach from \Cref{alg:greedy} to select $13$ sensor
locations out of the candidate sensors depicted in
\Cref{fig:num_res_one}~(right) that minimize a specified design criteria. 
For fast evaluations of the criteria in each iteration, we utilize the frozen
low-rank approximation approach from \Cref{sec:computational_methods} (see
\Cref{proc:frozen}). The resulting classical A-optimal and control-oriented optimal
designs are displayed in \Cref{fig:num_res_designs}. 

\begin{figure}[ht]
   \begin{subfigure}{.5\textwidth}
      \centering
      \includegraphics[width=1.1\textwidth]{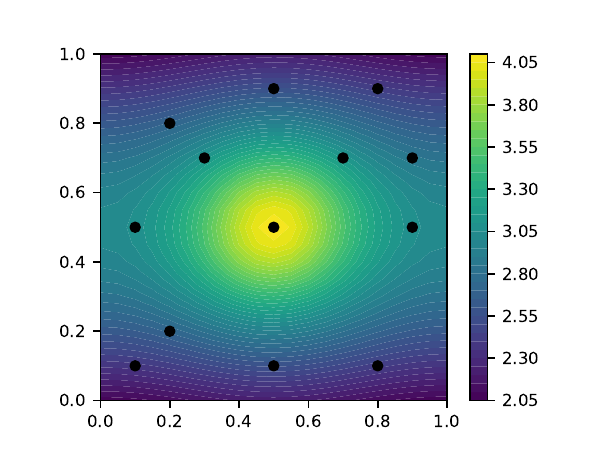}
      \caption{Classical (min. $\tr(\postcov)$)}
      \label{fig:num_res_design_classical}
   \end{subfigure}
   \begin{subfigure}{.5\textwidth}
      \centering
      \includegraphics[width=1.1\textwidth]{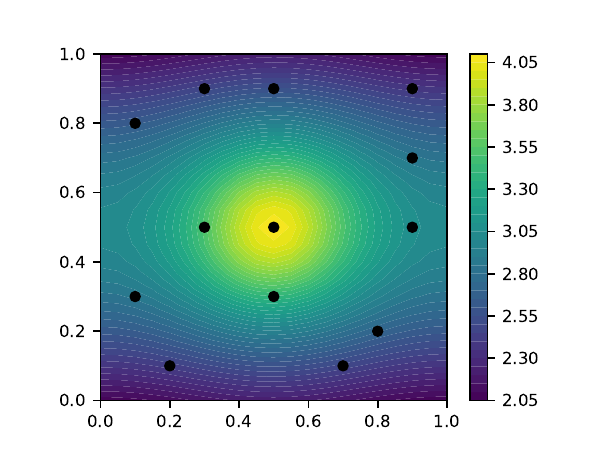}
      \caption{Control-oriented (min. $\tr(\AA\postcov\AA^*)$)}
      \label{fig:num_res_design_coed}
   \end{subfigure}
   \caption{Optimal experimental designs for the heat transfer problem.}
   \label{fig:num_res_designs}
\end{figure} 

Recall that the classical A-optimal criterion minimizes average posterior
variance of the inversion parameter. 
On the other hand, as
discussed in \Cref{sec:dataDrivenOC}, the control-oriented approach instead
seeks a reconstruction that minimizes uncertainty in the terminal state. In \Cref{tab:num_res_oed_metrics}, we compare the
measures of posterior uncertainty induced by utilizing classical and
control-oriented optimal designs. 
The reported results demonstrate that the control-oriented optimal placement of
sensors can yield significant benefits over classical approaches.  In
particular, we observe that when using cOED, the
average posterior variance of the terminal state is 19\% smaller than that
obtained using a classical A-optimal sensor placement.

\begin{table}[ht]
   \centering
   \def\arraystretch{1.25}
   \begin{tabular}{lrr}
   \toprule
   \textbf{Design}  & $\tr[\postcov]$  & $\tr[\AA\postcov\AA^*]$ \\ 
   \midrule
   Classical        & $\mathbf{0.590}$ & $2.26\times10^{-3}$  \\   Control-oriented & 0.616            & $\mathbf{1.83\times10^{-3}}$\\
   \bottomrule
   \end{tabular}
   \caption{Comparison of classical A-optimal and control-oriented sensor placements.}
   \label{tab:num_res_oed_metrics}
\end{table}

\begin{figure}[ht]\centering
\includegraphics[width=.5\textwidth]{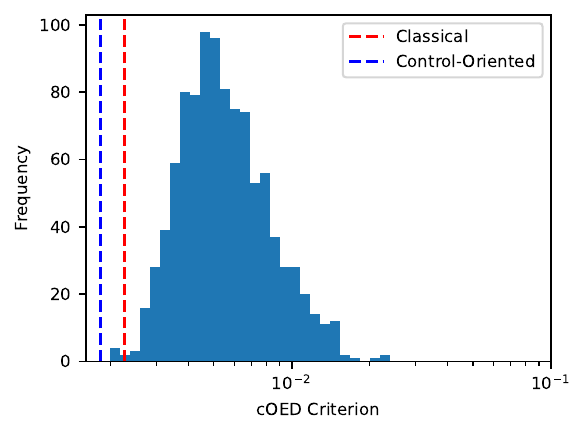}
\caption{Comparison of cOED criterion for random, classical and control-oriented designs.}
\label{fig:num_res_hists}
\end{figure}
It is also informative to consider the optimized designs against random sensor
selections. To this end, we generate a sample of 1000 random sensor placements
with the same number of sensors and compare them to the optimized designs; 
see \Cref{fig:num_res_hists}. Note that although there are combinatorially many 
possible design choices, these random sensor placements serve as a feasible and representative sample for comparison.
The results demonstrate that 
both classical and control-oriented designs outperform 
nearly all random designs. 
Furthermore, as noted previously, we observe that the control-oriented design 
outperforms the classical A-optimal design. It is worth highlighting that the benefit of cOED can be much larger in many practical applications. For example, applying cOED to even the same control problem studied here but equipped with a slightly stronger advection field can yield much larger (e.g., $>50\%$) improvement in the cOED criterion compared to a classical A-optimal design. As seen in \Cref{fig:num_res_hists_v2}, we observe roughly $60\%$ improvement when the advection field $\bs{v}(\bs{x})$ is scaled by a factor of five. Moreover, the performance also varies depending on the sensor budget $k$ chosen. To analyze the effect of the sensor budget, we compare the cOED criterion evaluated at the control-oriented and classical designs for varying sensor budgets in \Cref{fig:num_sensors}. For better visualization, we display the relative difference of the cOED criterion with sensor budget $k$ to the criterion evaluated with all candidate sensors active.

\begin{figure}[ht]
   \begin{subfigure}{.45\textwidth}
      \centering
      \includegraphics[width=1\textwidth]{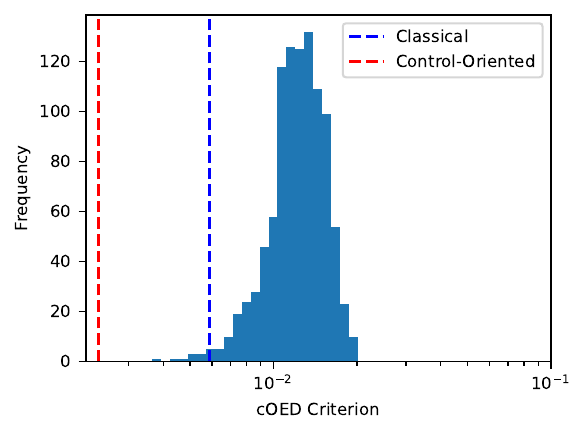}
      \caption{Histogram over random designs}
      \label{fig:num_res_hists_v2}
   \end{subfigure}
   \hspace{1em}
   \begin{subfigure}{.45\textwidth}
      \centering
      \includegraphics[width=1\textwidth]{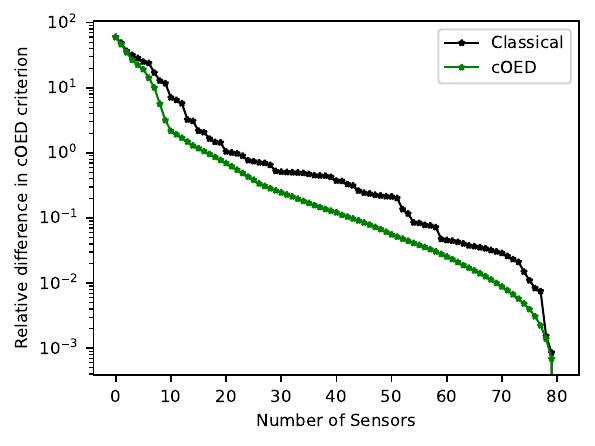}
      \caption{Comparison over sensor budget $k$}
      \label{fig:num_sensors}
   \end{subfigure}
   \caption{Comparison of cOED, classical, and random designs (scaled advection).}
   \label{fig:scaled_advection_case}
\end{figure}

\subsubsection{Control-Oriented Uncertainty Quantification}
\label{sec:num_res_oed_uq}
Here, we further study the effectiveness of control-oriented 
sensor placements in reducing the uncertainty in the terminal state.
In \Cref{fig:pointwise_var_coed},
we report the 
pointwise posterior variance field of the terminal state; this is compared with an evenly-spread sensor placement \Cref{fig:pointwise_var_rand_0}, as well as two random designs with an equivalent number of sensors
\Cref{fig:pointwise_var_rand_1,fig:pointwise_var_rand_2}. With both random and evenly-spaced sensor placements, we observe elevated amounts of posterior uncertainty in parts of the domain compared to cOED.
The results indicate a significant reduction in the posterior uncertainty with the
control-oriented approach.

\begin{figure}[ht]
   \centering
   \hspace{3em}
   \begin{minipage}{0.8\textwidth}
      \begin{subfigure}{.35\textwidth}
         \centering
         \includegraphics[width=\textwidth]{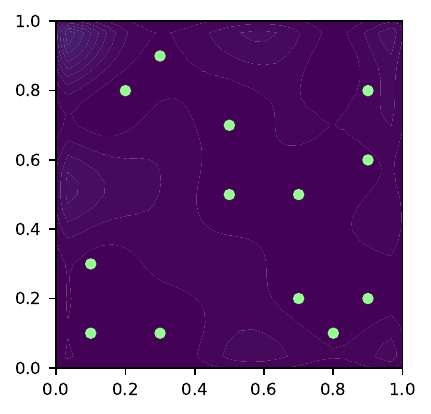}
         \caption{}
         \label{fig:pointwise_var_coed}
      \end{subfigure}
      \hspace{3em}
      \begin{subfigure}{.35\textwidth}
         \centering
         \includegraphics[width=\textwidth]{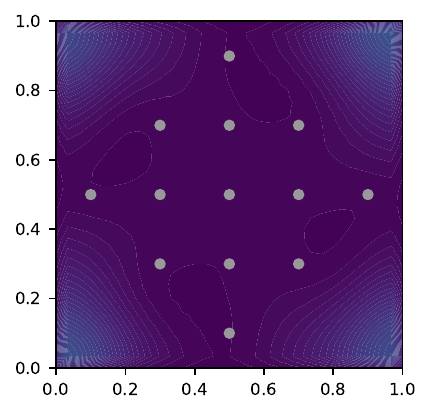}
         \caption{}
         \label{fig:pointwise_var_rand_0}
      \end{subfigure}
      \\
      \begin{subfigure}{.35\textwidth}
         \centering
         \includegraphics[width=\textwidth]{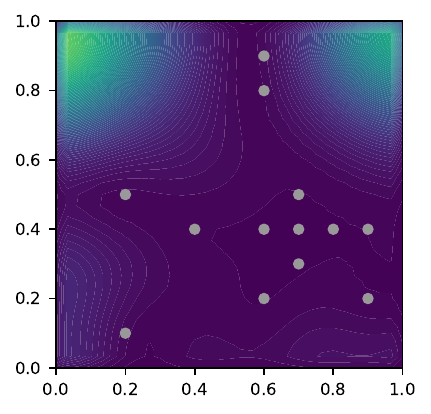}
         \caption{}
         \label{fig:pointwise_var_rand_1}
      \end{subfigure}
      \hspace{3em}
      \begin{subfigure}{.35\textwidth}
         \centering
         \includegraphics[width=\textwidth]{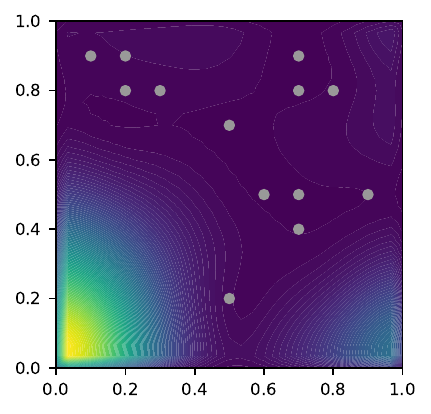}
         \caption{}
         \label{fig:pointwise_var_rand_2}
      \end{subfigure}
   \end{minipage}
   \hspace{-3em}
   \begin{subfigure}{.1\textwidth}
      \centering
      \includegraphics[height=4\linewidth]{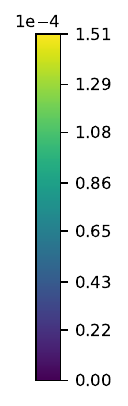}
      \vspace{-6em}
      \label{fig:pointwise_var_cbar}
   \end{subfigure}
   \caption{Posterior uncertainty in $\UT$ with cOED (a), evenly-spaced  designs (b), and random designs (c-d).}
   \label{fig:post_terminal_uncertainty}
\end{figure} 

We next consider 
the expectation and variance of the
control objective with respect to the posterior; see \Cref{sec:uq}. 
It is important to note that these posterior quantities rely on data and therefore must be interpreted with caution. In light of this, we consider a reachable target state $\ubar$ in the experiments that follow. Nevertheless, in the present experiment, cOED provides an expected control objective and a variance that is significantly better than
almost all evaluated random designs as seen in \Cref{fig:obj-exp-rand}. This suggests that cOED
can potentially be effective at aiding reduction of the control objective, 
as well as reduced uncertainty in the control objective. This is in
line with our theoretical observations in \Cref{sec:uq}.

\begin{figure}[ht]
   \centering
   \includegraphics[width=0.45\textwidth]{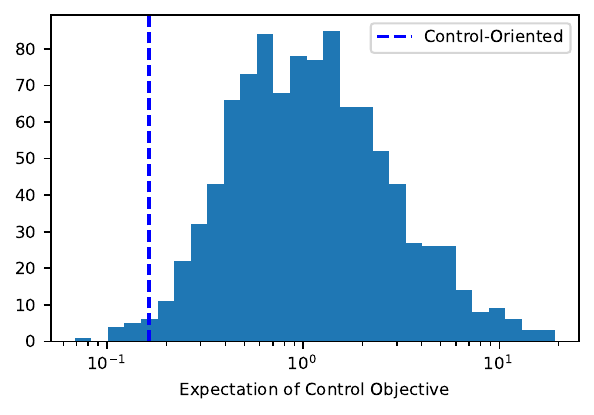}
   \quad
   \includegraphics[width=0.45\textwidth]{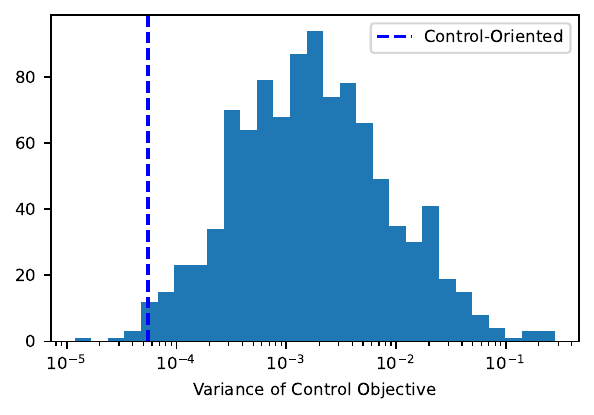}
   \caption{$\Exp(\ctrlObjs | \ \obs)$ (left) and $\Var(\ctrlObjs | \ \obs)$ (right) obtained 
   using the control-oriented design as well as random designs.}
   \label{fig:obj-exp-rand}
\end{figure}

\section{Conclusions}
\label{sec:conclusion}
In this article, we have developed a mathematical and computational framework
for control-oriented optimal sensor placement. The focus is on problems
involving control of linear PDEs with linear dependence on inversion parameters.
The proposed approach takes the coupling of the optimal control problem and the
associated inverse problem into account to derive a control-oriented
experimental design criterion.  This framework also enables quantifying the
uncertainty in the control objective.  The proposed mathematical framework is
complemented by scalable computational methods for solving the control-oriented
OED (cOED) problem and quantifying the uncertainty in the terminal state and the
control objective.  In particular, at the up-front cost of a modest number of
PDE solves, the proposed frozen low-rank forward operator approach enables
solving the cOED problem with no need for solving the PDEs
governing the inverse and optimal control
problems, within the optimization loop.

Our numerical results demonstrate the effectiveness of the proposed strategy and 
the importance of the control-oriented approach. Specifically, the control-oriented 
optimal sensor placements outperform classical A-optimal sensor placements 
in terms of reducing the posterior uncertainty in the terminal state 
in the PDEs governing the optimal control problem. Additionally, 
the posterior distribution obtained using the control-oriented optimal 
sensor placements provide near-optimal uncertainty reduction in the control 
objective. 

There are several interesting avenues for further investigations.  In the 
first place, one can
consider sensor placement for the cases where the optimal control problem has
nonlinear dependence on the inversion parameters or the cases where the associated
inverse problem is nonlinear.  For example, in the case of nonlinear inverse
problems, a Laplace approximation to the posterior provides a first step in
deriving approximate control-oriented measures of posterior uncertainty that are
tractable to optimize.  To provide more accurate estimates of control-oriented
measures of posterior uncertainty, sampling-based approaches are needed. To make
such sampling approaches tractable, derivative-informed surrogate-based approaches 
considered in~\cite{CaoRoseberryGhattas24,RoseberryChenVillaEtAl24,GoChen24} 
provide attractive options. Another interesting area for further investigation is 
to consider alternative approaches for defining a control-oriented criterion. 
These include 
a control-oriented D-optimality criterion, 
defined similarly to the goal-oriented setting in~\cite{Attia2021StochasticLA},
or a risk-adapted approach such as the one in~\cite{KouriJakemanHuerta22}.

\section*{Acknowledgments}
This article has been authored by employees of National Technology \& Engineering Solutions of Sandia, LLC under Contract No. DE-NA0003525 with the U.S. Department of Energy (DOE). The employees own all right, title and interest in and to the article and are solely responsible for its contents. The United States Government retains and the publisher, by accepting the article for publication, acknowledges that the United States Government retains a non-exclusive, paid-up, irrevocable, world-wide license to publish or reproduce the published form of this article or allow others to do so, for United States Government purposes. The DOE will provide public access to these results of federally sponsored research in accordance with the DOE Public Access Plan \url{https://www.energy.gov/downloads/doe-public-access-plan}. This paper describes objective technical results and analysis. Any subjective views or opinions that might be expressed in the paper do not necessarily represent the views of the U.S. Department of Energy or the United States Government. SAND2025-01903O.

\bibliographystyle{siamplain}
\bibliography{refs}

\appendix
\section{Discretization of the target state and control variables}
\label{appdx:disc}
The discretization of the target state $u_T \in \fnsp{U}$ can be performed 
similarly
to the case of the inversion parameter, using a finite element approach.  
On the other hand, the discretization of the
control variable $z$ is problem dependent. Consider, for example,  the case
where the control is a function of time only, as is the case in the example
in~\Cref{sec:motivation}. In that case, the discretized control variable
$\vec{z}$ is obtained by considering its values at a set of time steps $\{t_1,
\ldots, t_{\nz}\}$ in which case,  $\vec{z} = [z(t_1) \; z(t_2) \; \cdots
z(t_{\nz})]^\top$.  The inner product corresponding to the discretization of 
the space $\fnsp{Z}$ of controls is obtained similarly to the case of 
the inversion parameter described above. In the case where $\fnsp{Z} = L^2(0, T)$, 
the inner product of elements in $\fnsp{Z}$ can be approximated using a quadrature 
rule with nodes at the time steps and takes the form 
$\ip{\vec{z}_1}{\vec{z}_2}_{\Mt} \eqdef \vec{z}_1^\top \Mt \vec{z}_2$,
for $\vec{z}_1$ and $\vec{z}_2$ in $\R^{\nz}$. 
In this case, $\Mt$ is a diagonal matrix with quadrature weights on its
diagonal. The discretized control space is 
$\R^\nz$ equipped with the inner product $\ip{\cdot}{\cdot}_{\Mt}$. 
While the present definition concerns the case where
$\fnsp{Z} = L^2(0, T)$, the finite-dimensional inner product space of the
discretized control variable takes a similar form 
in general. 

\section{Proofs}\label{appdx:proofs}
\subsection{Proof of \Cref{prop:ctrlObjMeanVar}}
\label{sec:proof_ctrlObjMeanVar}
We first derive 
the expression for the first and second moments of 
a quadratic functional of a Gaussian random vector.
\begin{lemma}
   Consider a multivariate normal random vector, $\vec{x} \sim \GM{\vec{\mu}}{\mat{\Sigma}}$, where $\vec\mu \in \R^n$ and $\mat\Sigma \in \R^{n \times n}$ is symmetric positive-definite (SPD). Let $\spd$ be an SPD matrix. Then the quadratic form,
   \begin{equation}
   Q(\vec{x}) = \ip{\vec{x}}{\spd\vec{x}} = \|\vec{x}\|_{\spd}^2,
   \label{eq:quadDef}
   \end{equation}
   has 
   \begin{subequations}
      \begin{align}
      \Exp\left\{ Q(\vec{x}) \right\} &= \tr(\spd\mat{\Sigma}) + \vec{\mu}^\top \spd\vec{\mu},\\
      \Var\left\{ Q(\vec{x}) \right\} &= 2\tr\left[ (\spd\mat{\Sigma})^2 \right] + 4\vec{\mu}^\top \spd\mat{\Sigma}\spd\vec{\mu}.
      \end{align} 
   \end{subequations}
   \label[lemma]{lemma:QuadraticFuncs}
\end{lemma}
\begin{proof}
   Let $\vec{\hat{z}} \eqdef \mat{\Sigma}^{-1/2} (\vec{x}-\vec{\mu})$ be the standardized form of $\vec{x}$. One can express $\vec{x}$ in terms of $\vec{\hat{z}}$ as $\vec{x} = \mat{\Sigma}^{1/2}(\vec{\hat{z}} + \mat{\Sigma}^{-1/2}\vec{\mu})$.
Consider the spectral decomposition, $\mat{\Sigma}^{1/2}\spd\mat{\Sigma}^{1/2} = \mat{U}\mat{\Lambda}\mat{U}^\top$. Then, 
\begin{align*}
  Q(\mat{x}) = \vec{x}^\top \spd \vec{x}
  &= (\vec{\hat{z}} + \mat{\Sigma}^{-1/2}\vec{\mu})^\top \ {\Sigma}^{1/2}\spd\mat{\Sigma}^{1/2} \ (\vec{\hat{z}} + \mat{\Sigma}^{-1/2}\vec{\mu}) \\
  &= (\vec{\hat{z}} + \mat{\Sigma}^{-1/2}\vec{\mu})^\top \ \mat{U}\mat{\Lambda}\mat{U}^\top \ (\vec{\hat{z}} + \mat{\Sigma}^{-1/2}\vec{\mu}) \\
  &= (\mat{U}^\top\vec{\hat{z}} + \mat{U}^\top\mat{\Sigma}^{-1/2}\vec{\mu})^\top \mat{\Lambda}(\mat{U}^\top\vec{\hat{z}} + \mat{U}^\top\mat{\Sigma}^{-1/2}\vec{\mu}).
\end{align*}
For convenience, denote $\vec{z} \eqdef \mat{U}^\top\vec{\hat{z}}$, and $\vec{b} \eqdef \mat{U}^\top\mat{\Sigma}^{-1/2}\vec{\mu}$.
Then, 
\begin{align}
  Q(\vec{x}) = (\vec{z} + \vec{b})^\top \Lambda (\vec{z} + \vec{b}) = \sum_{i=1}^{n} \lambda_i(z_i + b_i)^2.
  \label{eq:chiSqForm}
\end{align}
Thus, the quadratic term follows a generalized chi-squared distribution, $$ Q(\vec{x}) \sim \tilde{\chi}^2\left( \{\lambda_i\}_1^n, \{d_i\}_1^n, \{b_i^2\}_1^n, 0, 0 \right)$$ where $d_i$ denotes the multiplicity of the $i$-th eigenvalue. Here, it is important to note that $z_i \sim \GM{0}{1}$ are independent standard normal random variables. This form can be used to obtain the mean and variance of $Q(\vec{x})$ as in \cref{eq:meanVarQuad}. These expressions are straightforward to derive (based on linearity) and we refer to \cite{Imhof61} and \cite{mathai1992quadratic} for an in-depth discussion of generalized chi-squared random variables:
\begin{align}
  \Exp\left\{ Q(\vec{x}) \right\} &= \sum_{i=1}^{n} \lambda_i (1+b_i)^2 = \tr(\spd\mat{\Sigma}) + \vec{\mu}^\top \spd\vec{\mu},\\
  \Var\left\{ Q(\vec{x}) \right\} &= 2\sum_{i=1}^{n} \lambda_i^2 (1+2b_i)^2 = 2\tr\left[ (\spd\mat{\Sigma})^2 \right] + 4\vec{\mu}^\top \spd\mat{\Sigma}\spd\vec{\mu}.
  \label{eq:meanVarQuad}
\end{align}
\end{proof}

We are now ready to state the proof of~\Cref{prop:ctrlObjMeanVar}: 
\begin{proof}
This result follows directly from the expectation and variance of the quadratic functional, \Cref{lemma:QuadraticFuncs}. For convenience, we express quantities with respect to the Euclidean inner product. 
By the scaling property of random normal vectors, we have $\tilde{\vec{e}} \eqdef \AA \dpar\ + \ \BB \ctrln +\vec{q} \ - \ \ubar \sim \GM{\AA\dparmap^\obs + \BB \ctrln + \vec{q} - \ubar}{\AA\postcov\AA^*\Mu^{-1}}$, which lies in $\R^n$. Note that the covariance of $\tilde{\vec{e}}$ is symmetric and positive-definite by construction.
By~\Cref{lemma:QuadraticFuncs} and the cyclic property of traces, we have 
\begin{align*}
   \Exp(\| \tilde{\vec{e}}\|^2_\Mu)
   &= \tr[\AA\postcov\AA^*] + \|\AA\dparmap^\obs + \BB \ctrln + \vec{q} - \ubar \|^2_\Mu,\\
   \Var(\| \tilde{\vec{e}}\|^2_\Mu) &= 2\tr[(\AA\postcov\AA^*)^2] + 4\|\AA\dparmap^\obs + \BB \ctrln + \vec{q} - \ubar \|^2_\mat{\Mu\AA\postcov\AA^*}
\end{align*}
Scaling the random variable $\| \tilde{\vec{e}}\|^2_\Mu$ by $\frac12$ gives the desired result. 
\end{proof}

\subsection{Hanson-Wright inequality and corollaries}

A useful tool for obtaining concentration bounds for sub-Gaussian random variables is the Hanson-Wright Inequality \cite{HansonWright1971,RudelsonVershyninHW2013}. We focus primarily on the application of this theorem to Gaussian random variables, which we state in \Cref{thm:HansonWright}.
\begin{theorem}[Hanson-Wright Inequality: Special Case]
   \label{thm:HansonWright}
   Let $\hat{\vec{z}} \sim \GM{\vec{0}}{\mat{I}_n}$ be a standard Gaussian random vector, and let $\spd$ be a symmetric matrix. Then, the quadratic random variable, $Q(\hat{\vec{z}}) = \ip{\hat{\vec{z}}}{\spd\hat{\vec{z}}}$ satisfies,
   \[\P(|Q(\hat{\vec{z}}) - \Exp\{Q(\hat{\vec{z}})\}| \geq t) \leq 2 \exp\left[ - \frac{1}{8} \min\left\{ \frac{t}{\|\spd\|_2}, \frac{t^2}{\|\spd\|_F^2} \right\}\right]\]
   for all $t \geq 0$.
\end{theorem}
\begin{proof}
From~\cite[Lemma 4]{CortinovisIndefinite2022}
\[ \begin{aligned}\P(|Q(\hat{\vec{z}}) - \Exp\{Q(\hat{\vec{z}})\}| \geq t) \leq &  2 \exp\left[ -  \frac{t^2 }{4t\|\spd\|_2 + 4\|\spd\|_F^2} \right] \\
\le & \> 2 \exp\left[ -  \frac{t^2 }{8\max\{t\|\spd\|_2 ,\|\spd\|_F^2\}} \right] . \end{aligned}\]
The proof is completed by a simple manipulation of the upper bound.
\end{proof}
\begin{remark}
Note that the constant $1/8$ in \cref{thm:HansonWright} can be replaced by any constant less than $0.14$, as shown in \cite[Lemma 1]{Moshksar2024Absolute} and \cite[Proposition 1]{MoshksarRefining2025}. The full proof of the Hanson-Wright inequality for sub-gaussian random variables can be referred to in \cite{RudelsonVershyninHW2013}. An alternative proof pertaining to the special case of Gaussian random vectors can also be found in the proof of \cite[Theorem B.8]{GiraudIntroduction2014}.
\end{remark}
\begin{theorem}[Hanson-Wright Inequality: General Covariance]
   \label{thm:HansonWrightSigma}
   Let $\vec{z} \sim \GM{\vec{0}}{\mat{\Sigma}}$ be a Gaussian random vector where $\mat{\Sigma} \in \R^{n\times n}$ is a symmetric positive-definite covariance operator. Consider a matrix $\spd \neq \mat{0}$ and the induced quadratic random variable, $Q(\vec{z}) = \ip{\vec{z}}{\spd\vec{z}}$. Then,
   \[\P(|Q(\vec{z}) - \Exp\{Q(\vec{z})\}| \geq t) \leq 2 \exp\left[ - \frac{1}{8} \min\left\{ \frac{t}{\|\mat{\Sigma}^{1 / 2}\spd\mat{\Sigma}^{1 / 2}\|_2}, \frac{t^2}{\|\mat{\Sigma}^{1 / 2}\spd\mat{\Sigma}^{1 / 2}\|_F^2} \right\}\right]\]
   for all $t \geq 0$. 
\end{theorem}
\begin{proof}
   Let $\hat{\vec{z}} \eqdef \mat{\Sigma}^{-1/2}\vec{z}$ and $\hat{\spd} \eqdef \mat{\Sigma}^{1/2}\spd\mat{\Sigma}^{1/2}$. Then, $\ip{\vec{z}}{\spd\vec{z}} = \ip{\hat{\vec{z}}}{\hat{\spd}\hat{\vec{z}}}$, so ${Q}({\vec{z}}) = \hat{Q}(\hat{\vec{z}})$. Applying \Cref{thm:HansonWright} to the standardized random vector and the transformed matrix gives the desired result.
\end{proof}
\begin{remark}
   Note that in \Cref{thm:HansonWrightSigma}, we can reformulate the norms to obtain a weaker but more comprehensible bound. In particular, note that by symmetry and cyclical commutativity of the trace,
   \begin{align*}
      \|\mat{\Sigma}^{1 / 2}\spd\mat{\Sigma}^{1 / 2}\|_F^2 = \tr\left[  (\mat{\Sigma}^{1 / 2}\spd\mat{\Sigma}^{1 / 2})^2\right] 
      = \tr\left[\mat{\Sigma}^{1 / 2}\spd\mat{\Sigma} \spd\mat{\Sigma}^{1 / 2}\right] 
      = \tr\left[(\mat{\Sigma}\spd)^2\right].
   \end{align*}
   Since the two-norm is bounded above by the Frobenius norm, we have that 
   \[\|\mat{\Sigma}^{1 / 2}\spd\mat{\Sigma}^{1 / 2}\|_2 \leq \|\mat{\Sigma}^{1 / 2}\spd\mat{\Sigma}^{1 / 2}\|_F = \sqrt{ \tr\left[(\mat{\Sigma}\spd)^2\right] }.\]
   Inserting these bounds into the Hanson-Wright Inequality, we obtain the following alternative bound,
   \begin{equation}
      \P(|Q(\vec{z}) - \Exp\{Q(\vec{z})\}| \geq t) \leq 2 \exp\left[ - \frac{1}{8} \min\left\{ \frac{t}{\sqrt{\tr\left[(\mat{\Sigma}\spd)^2\right]}}, \frac{t^2}{\tr\left[(\mat{\Sigma}\spd)^2\right]} \right\}\right]
      \label{eq:easierHW}
   \end{equation}
   for all $t \geq 0$. 
\end{remark}
\begin{lemma}
\label[lemma]{lemma:quad_concentration}
Consider $Q(\vec{x})$ from \cref{eq:quadDef}. Then, for any $t > 0$,
\begin{equation}
   \P(|Q(\vec{x}) - \Exp[Q(\vec{x})]| \geq t) \leq 4\exp\left[ - \frac{1}{8}\min\left\{ \frac{t}{\sqrt{\tr\left[(\mat{\Sigma}\spd)^2\right]}}, \frac{t^2}{\tr\left[(\mat{\Sigma}\spd)^2\right]}, \frac{t^2}{(\spd\vec{\mu})^\top\mat{\Sigma}(\spd\vec{\mu})} \right\} \right].
   \label{eq:quad_concentration}
\end{equation}
\end{lemma}
\begin{proof}
   Let $\vec{z} \eqdef \vec{x} - \vec{\mu}$ be the centered random variable in $\GM{\vec{0}}{\mat{\Sigma}}$. Then,
   \begin{align*}
     Q(\vec{x}) = \ip{\vec{x}}{\spd\vec{x}} &= \ip{\vec{z} + \vec{\mu}}{\spd(\vec{z} + \vec{\mu})} \\
     &= Q(\vec{z}) + 2\ip{\vec{z}}{\spd\vec{\mu}} + \ip{\vec{\mu}}{\spd\vec{\mu}}.
   \end{align*}
   By \cref{eq:easierHW}, note that for any $t \geq 0$,
   \[\P(|Q(\vec{z}) - \tr(\mat{A\Sigma})| \geq t) \leq 2 \exp\left[ - \frac{1}{8} \min\left\{ \frac{t}{\sqrt{\tr\left[(\mat{\Sigma}\spd)^2\right]}}, \frac{t^2}{\tr\left[(\mat{\Sigma}\spd)^2\right]} \right\}\right].\]
   Moreover, note that $\ip{\vec{z}}{\spd\vec{\mu}}$ is a Gaussian random variable with mean zero and variance $(\spd\vec{\mu})^\top\mat{\Sigma}(\spd\vec{\mu})$. Therefore, 
   by a Chernoff bound \cite{Chernoff1952Measure}, we have that
   \[\P(|\ip{\vec{z}}{\spd\vec{\mu}}| \geq t) \leq 2\exp\left[-\frac{t^2}{2(\spd\vec{\mu})^\top\mat{\Sigma}(\spd\vec{\mu})} \right].\]
   Using the triangle inequality and a union bound, we can combine the tail bounds to obtain that
   \begin{align*}
     \P(|Q(\vec{x}) - \Exp[Q(\vec{x})]| \geq t) &= \P(|Q(\vec{x}) - \tr(\spd) - \vec{\mu}^\top\spd\vec{\mu}| \geq t) \\[0.5em]
     &= \P(|Q(\vec{z}) + 2\ip{\vec{z}}{\spd\vec{\mu}} - \tr(A)| \geq t) \\[0.5em]
     &\leq \P(|Q(\vec{z}) - \tr(\spd)| \geq t) + \P(|2\ip{\vec{z}}{\spd\vec{\mu}}| \geq t)\\[0.5em]
     &\leq 2 \exp\left[ - \frac{1}{8} \min\left\{ \frac{t}{\sqrt{\tr\left[(\mat{\Sigma}\spd)^2\right]}}, \frac{t^2}{\tr\left[(\mat{\Sigma}\spd)^2\right]} \right\}\right] \\
     &\qquad+ 2\exp\left[-\frac{t^2}{8(\spd\vec{\mu})^\top\mat{\Sigma}(\spd\vec{\mu})} \right].
   \end{align*}
   We can combine the two exponential summands by their maximum, which provides the desired result.
\end{proof}

\subsubsection*{Proof of \Cref{prop:ctrlObjConcentration}}
\label{sec:proof_ctrlObjConcentration}
This result follows from the concentration bound for quadratic forms, \Cref{lemma:quad_concentration}, which builds on the Hanson-Wright Inequality. As in the proof of \Cref{prop:ctrlObjMeanVar}, we consider $\tilde{\vec{e}} \eqdef \AA \dpar\ + \ \BB \ctrln +\vec{q} \ - \ \ubar \sim \GM{\AA\dparmap + \BB \ctrln + \vec{q} - \ubar}{\AA\postcov\AA^*\Mu^{-1}}$ 
so that $\ctrlObjs = \|\tilde{\vec{e}}\|_\Mu^2$. 
Since $\AA\postcov\AA^*\Mu^{-1}$ is symmetric positive-definite, \Cref{lemma:quad_concentration} yields that for any $t \geq 0$, we have
\begin{align*}
   \P\left( \left| \|\tilde{\vec{e}}\|_\Mu^2 -  \Exp(\|\tilde{\vec{e}}\|_\Mu^2)\right| \geq \tau\right) \leq 4\exp\left[ - \frac{1}{8}\min\left\{ \frac{\tau}{\sqrt{\tr\left[(\AA\postcov\AA^*)^2\right]}}, \frac{\tau^2}{\tr\left[(\AA\postcov\AA^*)^2\right]}, \frac{\tau^2}{C^2} \right\} \right],
\end{align*}
where $C$ is as defined in the statement.
Recall that for a symmetric positive semi-definite matrix $\mat{Z}$, 
\(\tr[\mat{Z}^2] \leq \tr[\mat{Z}]^2\).
Applying this result to $\mat{Z} \eqdef \Mu^{1 / 2}\AA\postcov\AA^*\Mu^{-1 / 2}$, we have
\begin{equation}
   \tr[(\Mu^{1 / 2}\AA\postcov\AA^*\Mu^{-1 / 2})^2] \leq [\tr(\Mu^{1 / 2}\AA\postcov\AA^*\Mu^{-1 / 2})]^2.
   \label{eq:trSqPropMass}
\end{equation}
Notice that the trace is invariant under similarity transformation, i.e.,
\[\tr(\Mu^{1 / 2}\AA\postcov\AA^*\Mu^{-1 / 2}) = \tr(\AA\postcov\AA^*)\]
and 
\[\tr[(\Mu^{1 / 2}\AA\postcov\AA^*\Mu^{-1 / 2})^2] = \tr[\Mu^{1 / 2}(\AA\postcov\AA^*)^2\Mu^{-1 / 2}] = \tr[(\AA\postcov\AA^*)^2].\]
Thus, \cref{eq:trSqPropMass} simplifies to \(\tr[(\AA\postcov\AA^*)^2] \leq [\tr(\AA\postcov\AA^*)]^2\). Substituting this inequality into the above probability bound achieves the desired result.\hfill\proofbox

\subsubsection*{Proof of \Cref{cor:controlObjInterval}}
\label{sec:proof_ctrlObjConcentrationCorollary}
   The proof follows from a simple manipulation of the concentration bound \cref{eq:ctrlObjConcentration}. 
   Namely, we want to obtain $\tau$ such that the upper bound for the deviation probability in \cref{eq:ctrlObjConcentration} is less than $\delta$, i.e., such that
   \[4\exp\left[ - \frac{1}{8}\min\left\{ \frac{\tau}{\Psi^{cA}}, \frac{\tau^2}{(\Psi^{cA})^2}, \frac{\tau^2}{C^2} \right\} \right] \leq \delta.\]
   Rearranging and solving for $\tau$ yields the upper bound in \cref{eq:ctrlObjConcentrationInterval}. \hfill\proofbox

\end{document}